\documentclass[10pt,a4paper]{article}
\usepackage[latin1]{inputenc}
\usepackage[T1]{fontenc}
\usepackage{amsmath, amsthm}
\usepackage{amsfonts}
\usepackage{amssymb}
\usepackage{bbold,mathtools}
\usepackage{graphicx}
\usepackage{color}
\usepackage[margin=1.25in]{geometry}
\usepackage{enumerate}
\usepackage[symbol,perpage]{footmisc}
\usepackage{authblk}

\newtheorem{theorem}{Theorem}
\newtheorem{proposition}[theorem]{Proposition}
\newtheorem{lemma}[theorem]{Lemma}
\newtheorem{conjecture}[theorem]{Conjecture}

\numberwithin{theorem}{section}

\newcommand{\dd}{\mathrm{d}}
\newcommand{\cha}{\mathbb{1}}
\newcommand{\eps}{\varepsilon}

\newcommand{\R}{\mathbb{R}}
\newcommand{\N}{\mathbb{N}}

\DeclareMathOperator{\sign}{sign}

\DeclareMathOperator{\inte}{int}
\DeclareMathOperator{\id}{id}
\DeclareMathOperator{\Leb}{Leb}
\DeclareMathOperator{\im}{im}
\newcommand{\Wul}{\mathcal{W}^u_{\rm loc}}
\newcommand{\C}{\mathcal{C}}
\newcommand{\D}{\mathcal{D}}
\renewcommand{\P}{\mathcal{P}}
\newcommand{\M}{\mathcal{M}}
\renewcommand{\S}{\mathcal{S}}
\newcommand{\condmeas}{\rho_\S}
\newcommand{\slice}{\empty}

\newcommand{\vv}{v}
\newcommand{\vu}{v^u}
\newcommand{\vs}{v^s}

\newcommand{\B}{\mathcal{B}}
\renewcommand{\H}{\mathbf{H}}
\newcommand{\Ht}{\H}%{\H^{t,s}_p}}

\renewcommand{\L}{\mathcal{L}}

\renewcommand{\O}{\mathcal{O}}
\newcommand{\cyi}{\mathbf{i}}
\newcommand{\cyj}{\mathbf{j}}

\newcommand{\bh}{_{(\beta)}}
\newcommand{\bhh}{_{(\beta+1)}}
\newcommand{\bg}{_{(\alpha)}}
\newcommand{\bgg}{_{(\alpha+1)}}
\newcommand{\bgh}{_{(\alpha,\beta)}}

\newcommand{\Aord}{1} % the level of differentiability required for observable A

\author{Caroline L. Wormell\thanks{Laboratoire de Probabilit\'es, Statistique et Mod\'elisation,
		Sorbonne Universit\'e, CNRS \\
		email: {\sf wormell@lpsm.paris}, {\sf ca.wormell@gmail.com}\\
		ORCID: 0000-0003-2953-6493}
}
\title{On convergence of linear response formulae in some piecewise hyperbolic maps} %Convergence of linear response formula for some piecewise hyperbolic maps}
\begin{document}
	\maketitle
	
	\begin{abstract}
		When high-dimensional non-uniformly hyperbolic chaotic systems undergo dynamical perturbations, their long-time statistics are generally observed to respond differentiably with respect to the perturbation. Although important in applications, this differentiability, which is thought to be connected to the dimensionality of the system, has remained resistant to rigorous study. 
		
		To model non-uniformly hyperbolic systems, we consider a family of the mathematically tractable class of piecewise smooth hyperbolic maps, the Lozi maps. For these maps we prove that the existence of a formal derivative of the response reduces to an exponential mixing property of the SRB measure when conditioned on the map's singularity set. This property appears to be true and is of independent interest. Further study of this conditional mixing property may yield a better picture of linear response theory.
	\end{abstract}
%	\red{\it Mixing from conditional measures}
%	\\

%	\red{Maybe motivate studying long term behaviour of fractal sets somewhere too. For example, initial conditions set on manifold of dimension lower than unstable dimension, or initial conditions sampled from SRB measure but conditioned on some observation.}
	
Because of their irregular behaviour at a trajectory level, chaotic systems are largely studied in terms of their invariant measures. A widespread and physically relevant class of invariant measures are the so-called Sinai-Ruelle-Bowen (SRB) measures, which measure the proportion of time that, over the long term, a Lebesgue-generically initialised orbit will spend in any region of phase space. For example, in an atmospheric system, the SRB measure encodes climatic probabilities. One may therefore naturally ask how the SRB measures respond when the parameters of the system are changed. 

From this question the so-called linear response theory has developed, which aims to answer this question to first order. In an autonomous setting we might imagine, on some manifold $\M$, a one-parameter family of maps $f^\eps: \M \to \M$ with 
\begin{equation}\label{eq:Perturbation}
f^\eps(x) = f(x) + \eps X(f(x)) + o(\eps),
\end{equation}
each possessing an SRB measure $\rho^\eps$. The aim of linear response theory is to find a derivative of $\rho^\eps$ at $\eps = 0$, under the assumption that such a derivative exists, which is not always the case \cite{Baladi14}. 

When the system is mixing, the derivative can be expressed formally as a sum
\begin{equation}\label{eq:LinearResponseFormula}
\left.\frac{\dd}{\dd\eps} \int A\,\dd\rho^\eps \right|_{\eps = 0} = \sum_{n=0}^\infty \kappa_n,
\end{equation}
where the susceptibility coefficients are defined as
\begin{equation}\label{eq:Kappa}
\kappa_n = \int_\M \nabla(A\circ f^n)\cdot X\,\dd\rho 
\end{equation}
where $\rho := \rho^0$ is the unperturbed system's SRB measure. (An attraction of linear response theory is that this can be computed purely in terms of the unperturbed system.) 
When the susceptibility coefficients are summable, the derivative in \eqref{eq:LinearResponseFormula} typically exists in all examples where we have mathematical proofs of absolute convergence \cite{Baladi14}\footnote{Finiteness of analytic continuations of the susceptibility function $\Psi(z) = \sum_{n=0}^\infty \kappa_n z^n$ at $z=1$ is more suspect: see \cite{Jiang05} for an example where this $\Psi(z)$ continues to a bounded function at $z=1$, but linear response but does not obtain \cite{Baladi14}.}.

Nonetheless, the integrand in \eqref{eq:Kappa} can be expected to explode as $n\to\infty$, so one must apply something more to obtain a convergent sum, such as knowledge of the map's geometry. When the map is, say, a uniformly hyperbolic, $C^3$ diffeomorphism, this may be achieved by projecting $X$ in stable and unstable directions, and performing an integration by parts in \eqref{eq:Kappa} along unstable manifolds to secure convergence of the unstable component \cite{Ruelle97}. This integration by parts works because the SRB measure is (by definition) conditionally absolutely continuous along unstable manifolds, and in fact for uniformly hyperbolic diffeomorphisms its density is smooth with good control of derivatives.

On the other hand, uniformly hyperbolic maps are rather atypical among chaotic systems: in applications one is often confronted with tangencies between stable and unstable manifolds. In this case, the uniform splitting between stable and unstable directions cannot be applied and the unstable density contains unpleasant singularities. As a result, even using a particular generous interpretation of regularity, the map $\eps \mapsto \rho^\eps$ is no better than $1/2$-H\"older for generic families of quadratic maps, which model stable-unstable tangencies in one dimension \cite{Baladi15}. The reason for this is that singularities in the map generate singularities in the SRB measure along unstable manifolds, whose derivatives blow up when pushed forward by the map. This rationale would predict $C^{0.5^-}$ response for all maps with homoclinic tangencies: however, most sufficiently high-dimensional (non-uniformly hyperbolic) maps appear to have a linear response \cite{GallavottiCohen95a, WormellGottwald19}. Indeed, the linear response formula in one guise or another has been widely applied in the physical sciences.

In the last decade a hypothesis for this mismatch has been put forth by Ruelle \cite{Ruelle11,Ruelle18}: the map's singularities are supported on a cross-section of the attractor and thus form a fractal set. Intuitively, this set has a positive dimension, given by the stable dimension $d_s$. If the stable dimension $d_s$ is greater than $3/2$, the singularities in the measure, when projected onto a single unstable manifold, will therefore average out to an absolutely continuous measure, and this measure's derivative decays exponentially when pushed forward by the map, thus inducing linear response in this unstable map. Generalising this further, the response of such a non-uniformly hyperbolic system should generically be a little worse than $C^{d_s - 1/2}$ (in fact, almost-$C^{d_s + 1/2}$ regularity is claimed under optimistic claims on the exponential rate of mixing). However, this conjecture has seen little progress, in no small part because families of non-uniformly hyperbolic systems in more than one dimension are notoriously difficult to study.
\\

In this paper, therefore, we attempt to make progress through a rigorous study of a pertinent simpler class of maps, those of piecewise uniformly hyperbolic maps. The edges of the smooth pieces of these maps model tangencies between stable and unstable manifolds. Compared with non-uniformly hyperbolic maps they are much easier to deal with: already in linear response theory, their one-dimensional equivalents have been used as precursors to the logistic map \cite{Aspenberg19}. Furthermore, a robust theory of mixing of piecewise hyperbolic maps has already been developed \cite{DemersLiverani08, BaladiGouezel10}.

A very simple example of a piecewise hyperbolic map with a positive stable dimension is the Lozi map on $\R^2$ \cite{Lozi78, Misiurewicz80, ColletLevy84, Young85}
\begin{align}
	f(x,y) = (1 + y - a|x|, bx),\, b \neq 0 \label{eq:LoziDef}
\end{align}
It is worth noting that by replacing $|x|$ by $x^2$, one recovers the non-hyperbolic H\'enon map. For $a \geq 0$, $b \in (0,\min\{a-\sqrt{2},4-2a\})$ this map is known to have a unique SRB measure with exponential mixing for $C^1$ observables.

The response of such Lozi maps has been shown to be $\alpha$-H\"older for all $\alpha<1$ \cite[Theorem~2.13]{DemersLiverani08}. In the limit case $b = 0$ (i.e. the stable dimension $d_s$ is zero and we reduce to the one-dimensional tent map), the Lozi map's response is generically at best log-Lipschitz\footnote{A function $\phi$ is log-Lipschitz at $s$ if $\|\phi(t)-\phi(s)\| = \mathcal{O}(|t-s| \log |t-s|)$ as $t \to s$.}, i.e. slightly worse than Lipschitz or $C^1$ \cite{Keller82, BaladiSmania08}. In light of the aforementioned conjecture on non-uniformly hyperbolic systems \cite{Ruelle18}, it is natural to ask whether increasing the stable dimension of piecewise hyperbolic systems (and in particularly the Lozi map) improves the regularity of the response from the $d_s = 0$ baseline of log-Lipschitz. Since Lozi maps have dimension strictly greater than the number of positive Lyapunov exponents \cite{Schmeling98}, and hence strictly positive stable dimension\footnote{{\it c.f.} the proof of Proposition~\ref{p:CriticalLineSelfIntersectionMeasure} in the Appendix, which shows that the transverse measure is atomless.}, 
this is to ask:
\begin{quote}
	\it Do SRB measures of the Lozi map have linear response in general?
\end{quote}

This paper's main result, Theorem~\ref{t:Formal}, reduces the question of a formal linear response for Lozi maps (or conceptually, any piecewise hyperbolic system with exponential mixing) to, up to small generalisation, the existence of sufficiently fast {\it conditional mixing} on the singular line \cite{mix}:
\begin{quote}
	\it Let $\rho$ be the SRB measure of a Lozi map $f$, let $\condmeas$ be a scalar multiple of the (suitably defined) conditional measure of $\rho$ on the singularity set of $f$, and $A, B \in C^1(\R^2,\R)$. Is the following sequence exponentially decaying in $n$:
	\begin{equation} \int A \circ f^n\, B \dd\condmeas - \int A\,\dd\rho \int B \dd\condmeas.  \label{eq:CondMixing}\end{equation}
\end{quote}
This is to say that when the weighted conditional measure $B\condmeas$ is pushed forward under the map $f$, it (weakly) converges to the SRB measure, and does this exponentially quickly. 

In the preceding work \cite{mix} conjectured that \eqref{eq:CondMixing} holds for conditional measures of the SRB measures on generic lines, supported by careful numerical experiments on the singular line and strong analogous results in toy models. In this paper we conjecture (Conjecture~\ref{c:Full}) that exponential conditional mixing holds for the singular line in a specialised sense necessary for Theorem~\ref{t:Formal}. In Section~\ref{s:Results} of this paper, we also give numerical evidence indicating the existence and validity of the formal linear response.

The connection between exponential conditional mixing and linear response arises in the same way as in \cite{Ruelle18}, namely that the obstructions to linear response arise from
%The connection between \eqref{eq:Question} and linear response arises as follows: LRT is well-known to depend on mixing properties of the {\it derivative} of the measure in unstable directions \cite{BaladiSmania12,Baladi14}. The obstructions to fast mixing of the derivative of the measure are 
non-mixing of singularities in the smooth hyperbolic structure of the measure, which originate from that of the map. If these singularities mix quickly, then we expect linear response. %In fact, a generalisation of this argument to more complex singularities would yield linear response for non-hyperbolic systems \cite{Ruelle18}.

%Furthermore, there is a literature on a related problem: suppose one takes a (transformation of a) Gibbs invariant measure $\nu$ of a map $g$, and pushes it forward under a sufficiently unrelated map $f$, do statistical behaviours ordinarily associated with $f$'s SRB measure $\rho$ obtain? (The typically Lebesgue-singular measure $\nu$ we analogise with $\condmeas$.) For example, does a Birkhoff law on $f$ obtain for $\nu$-almost all initial conditions? Does $f^n_* g \to \mu$? When $f$ is piecewise linear and $g$ is smooth, some results are known: a Birkhoff law obtains ...; the theory of Fourier dimension implies that when $f(x) = kx \mod 1$ one obtains exponential mixing for various dynamical measures $\nu$, including ....} Thus, we should be open to the likelihood that the ``basin of attraction'' of a system's SRB measure contains a much larger set of interesting measures than the Lebesgue-regular measures usually studied in transfer operator theory.

Of course, we have merely related the question of (formal) existence of linear response for the Lozi map to the question of conditional mixing. While this latter question remains mathematically unresolved, we have clear evidence that it is true, it is seemingly of broader interest and may be answered in future.

We expect that in practice the regularity of the response is better than differentiable, and by analogy with the more optimistic statement in \cite{Ruelle18} would expect a slightly less than $C^{1+d_s}$ response, with $d_s$ the stable dimension of the map. An understanding of higher-order response for piecewise hyperbolic maps would be a step towards understanding the non-uniformly hyperbolic situation.

While we prove results only for the Lozi map we expect our results can be extended to more general piecewise hyperbolic maps: for nonlinear maps there is one extra, smooth, term which standard transfer operator results will show to decay exponentially. Our approach may also shed light on higher-order response for Sinai billiards, which possess similar piecewise-singularities in the dynamics for orbits that narrowly miss scatterers \cite{Chernov14}.
\\

Our paper is structured as follows: in Section \ref{s:Setup} we describe the Lozi map as much as necessary to state our main theoretical and numerical results in Section \ref{s:Results}. In Sections \ref{s:Sobolev} we introduce the existing decay of correlations result, and Section \ref{s:Disintegration} we describe various disintegrations of the SRB measure; in Section \ref{s:Decomposition} we decompose the susceptibility function into several components which we then pin down the behaviour of in the following Sections \ref{s:DecayX}--\ref{s:DecayLRho}.

\section{Piecewise hyperbolicity of the Lozi map}\label{s:Setup}
	
For $b\in(0,\min\{a-1,4-2a\})$ the Lozi map \eqref{eq:LoziDef} has at least one strange attractor $\Lambda$ (see Figure~\ref{fig:Cones}) contained in a strictly invariant bounded open set \cite{Misiurewicz80, Young85}, where $p_0 = (2/(2+a-\sqrt{a^2+4b}),0)$. We can choose a strictly invariant closed subset $\M$ of the open set such that $\Lambda \subset \inte\M$ is the unique attractor in $\M$, and note that we can draw $\M$ such that it is convex and its edges avoid the stable cone (defined below).%For $b\in(0,\min\{1+a,4-2a\})$ there exists a proper triangle $\Delta= \hull\{w_0,w_1,w_2\} \subset \R^2$ which is invariant under $f$. This triangle has an open neighbourhood that is strictly invariant under $f$ and which shrinks uniformly on and into $\Delta$. Let $\M$ be a strictly invariant closed set inside this open neighbourhood whose edges avoid the stable cone (defined below), possibly via zig zags, and whose interior contains $\Delta$. All the attractors of $\M$ are contained in $\Delta$. %Let $\tilde \Delta$ be a small expansion of this triangle (?).

%	 Since $w_{-1} \in \overrightarrow{w_0 w_1}$ we have that $\overrightarrow{w_0 w_1}$ and $\overrightarrow{w_0 w_2}$ lie in unstable directions, and we can choose a (possibly zig-zagging) line between $w_1$ and $w_2$ everywhere uniformly transverse to the stable cones we will define forthwith. Let $\M$ be the area enclosed by these three lines, and define $\cal{S} = \partial M \cup \ell_{\S}$ 

The Lozi map is piecewise affine. Let the domain of the two pieces of the map be given respectively by
\[ \M_{\pm} = (\R^{\pm} \times \R) \cap \inte\M. \]

The boundary between these two pieces is the critical line
\begin{equation}
\ell_{\S} = \{0\} \times \R.
\end{equation}
For $x \in \ell_{\S}$ we notate $x_\pm$ as $x$ considered respectively as being in $\M_\pm$.

The singular set of $\M$ is $\S := \ell_{\S} \cup \partial\M$, but the boundary component is neither interesting nor problematic to us, as it does not intersect with any attractors of $f$.

%	Let the forward (resp. reverse) orbit of $\S$ be
%	\begin{equation}
%	\S^{\pm} = \bigcup_{n=0}^{\infty} f^{\pm n}(\S)
%	\end{equation}
%The singularity sets $\S^+ = \ell_{\S} \cup \partial \M$ and $\S^- = f(\ell_{\S}) \cup \partial \M$. 

Let $c_{\rm cone}$ be less than $\frac{a+\sqrt{a^2-4b}}{2b}$ by some sufficiently small amount. For $a \geq 1+b$ there exist piecewise-$C^2$ invariant cones %{\it (need Mathematica for the constants)}
\begin{equation} \C^u(x,y) = \{ (\xi,\eta) \in \R^2 : \xi \geq c_{\rm cone} |\eta| \} \label{eq:UnstableCone} \end{equation} % 
and
\begin{equation}\label{eq:StableCone}
\C^s(x,y) = Df^{-1}_{(x,y)} \{ (\xi,\eta) \in \R^2 : \eta \geq c_{\rm cone}  |\xi| \}.
\end{equation}%-2a|w| < 2b |v| - (a+\sqrt{a^2-4b}+\epsilon)|w| < 2a|w| \}.\]
These cones are plotted in Figure~\ref{fig:Cones}.

\begin{figure}
	\centering
	\includegraphics{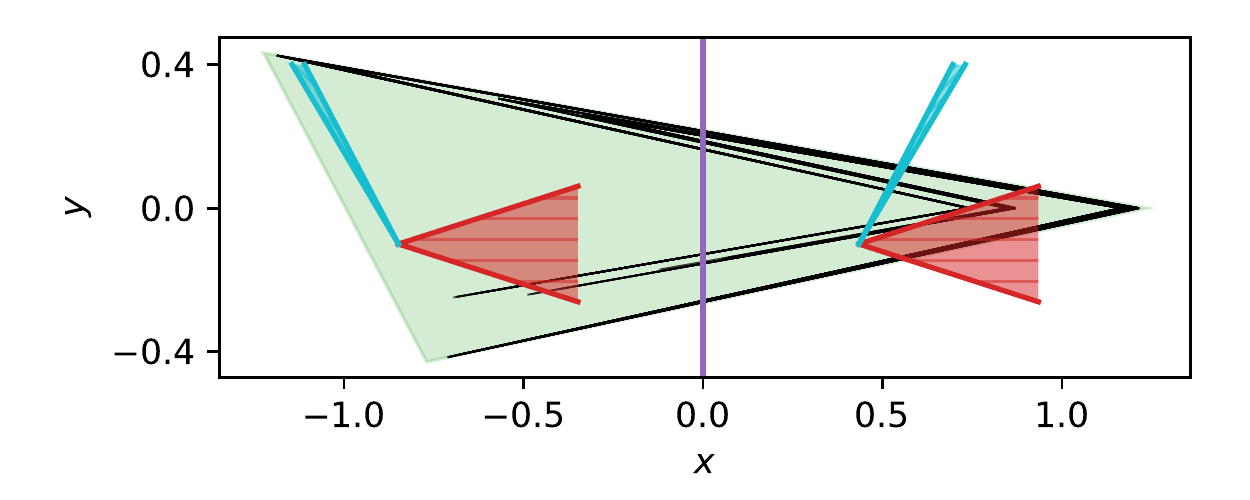}
	\caption{Picture of the Lozi attractor at $a=1.8, b=0.35$ (black) with an absorbing set $\M$ (green), the singular line $\ell_{\S}$ (purple), and the unstable (red, horizontal hatch) and stable (blue, diagonal hatch) cones on either side of the singular line.}
	\label{fig:Cones}
\end{figure}
	
These cones satisfy uniform expansion and contraction conditions. In the unstable direction, $Df_p \C^u(p) \subset \inte \C^u(f(p)) \cup \{0\}$ for all $p \in \M$, and there exists $\lambda > 1$ such that
\[ \|Df_p v\| \geq \lambda \|v\| \]
for all $v \in \C^u(p),\, p \in \M$. In the stable direction, $Df_{f(p)}^{-1} \C^s(f(p)) \subset \inte \C_S(p) \cup \{0\}$ for all $p \in \M$, and there exists $\mu < 1$ such that
\[ \|Df_{f(p)}^{-1} v\| \geq \mu^{-1} \|v\| \]
for all $v \in \C^s(f(p)),\, p \in \M$.
The tangent vectors to the singular line $\ell_{\S}$ are transverse to all stable cones, and as mentioned before it is possible to construct $\M$ so that the pieces of the boundary $\partial\M$ also are also all transverse to stable cones.

Let $\vu$ and $l^u$ be the unstable vector bundle (resp. covector bundle) of $f$, defined such that $\|\vu(x)\| \equiv l^u(x) \vu(x) \equiv 1$, and $\vu(x) \cdot e_1 >0$ (i.e. $\vu$ is always pointing to the right, and therefore always in $\C^u$). %We will assume that $r(x)$ is defined off the attractor by extending $f: \M \to \M$ back to the diffeomorphism $f: \R^2 \to \R^2$. %Along an unstable manifold, $l, r$ are discontinuous functions of bounded variation.
Similarly, let $\vs(x)$ and $l^s(x)$ be the stable vector bundle and covector bundle respectively such that $\| \vs(x) \| = l^s(x) \vs(x) = 1$, and $\vs(x) \cdot e_2> 0$ (i.e. $\vs$ is always in the stable cone $\C^s$). 

Furthermore let us define the pointwise rate of expansion along unstable manifolds 
\begin{equation}\lambda_n(x) := l^u(f^n(x)) D_xf^n \vu(x)\label{eq:LambdaNDefinition}\end{equation}
and the pointwise rate of contraction in the stable direction (in backwards time)
\[\mu_n(x) := l^s(x) D_{f^{-n}(x)}f^{n} \vs(f^{-n}(x)) = \det (D_{f^{-n}(x)}f^n) / \lambda_{n}(f^{-n}(x)).\] 
We have $| \lambda_{n}(x) | \geq \lambda^n > 1$ and $| \mu_n(x) | \leq \mu^n < 1$.

%We require three assumptions on our Lozi maps for the results in the sequel to hold. These assumptions hold on an open set of maps: in particular on the set 
%\[ \left\{(a,b): a \in (\sqrt{2},2), b \in (0,\min\{a - \sqrt{2},4-2a\})  \right\}.\]
%The first two are really complexity bounds, and we expect that the effect of these assumptions should hold for almost all Lozi(-like) maps.

\section{Main results}\label{s:Results}

%Following \cite{Young85}, the projected measure onto the $x$ component $\rho(\dd x \times \R)$ is absolutely continuous. In particular, this means that for each $x \in \R$ there exists a conditional (non-probability) measure $\rho(y\mid x)$ on $\{x\}\times \R$ such that
%\[ \rho(A) = \int_{\R} \int_\R A(x,y)\, \dd \rho(y\mid x)\, \dd x. \]
%Let us define $\condmeas = \rho(\cdot\mid0)$ to be the conditional measure on the singular line $\ell_{\S}$ (and hence on the singularity set $\S$). 

%We begin by showing that we can disintegrate the measure $\rho$ along lines of constant $x$ coordinate, and obtain meaningful ``slice'' measures $\rho_{\slice x}$ that exist {\it for all $x$}. (An abstract disintegration result would give existence only for almost all $x$.)
%
%This allows us to define the measure $\condmeas := \rho_{\slice 0}$ as the ``conditional'' measure on the singular set $\S$. Note however that these measures $\{\rho_{\slice x}\}$ are only scalar multiples of conditional measures of $\rho$ on the foliation $\{\ell_x\}_{x \in \R}$. We will henceforth refer to them as {\it slice measures}, as they represent infinitesimal ``slices'' of the full measure $\rho$.

In the first instance we prove that one can define something akin to a conditional measure of the SRB measure $\rho$ along {\it any} leaf of an appropriate foliation. This is necessary, because we are interested in a specific leaf (the singular line), and abstract results only give existence for almost all leaves. By contrast, the proof of this theorem will rely on the SRB measure's manifold structure.

\begin{theorem}\label{t:SliceMeasures}
	For each $x \in \R$ there exists a positive Borel measure $\rho_{\slice x}$ with support on $\ell_x := \{x\} \times \R$ such that:
	\begin{enumerate}[a.]
		\item \label{res:finite} $\sup_{x \in \R} \rho_{\slice x}(\ell_x) < \infty$;
		\item \label{res:slice} For all $x \in \R$ and all $A: \R^2 \to \R$ bounded and continuous,
		\begin{equation} \int_{\R^2} A\,\dd\rho_{\slice x} = \lim_{\delta \to 0} \frac{1}{2\delta} \int_{(x-\delta,x+\delta) \times \R} A\,\dd\rho. \label{eq:SliceMeasureProperty}\end{equation}
		\item \label{res:disintegration} For all measurable $E \subseteq \R^2$,
		\begin{equation*}
		\rho(E) = \int_{\R} \rho_{\slice x}(\ell_x\cap E)\,\dd x. %\label{eq:DisintegrationProperty}
		\end{equation*} 
	\end{enumerate}
	Furthermore, \eqref{eq:SliceMeasureProperty} specifies $\rho_{\slice x}$ uniquely.
\end{theorem}
Note that these measures are not themselves probability measures in general, but are scalar multiples of the conditional measures along the $\ell_x$. To disambiguate therefore, we call these measures $\rho_{\slice x}$ ``slice measures''.
This measures $\rho_{\slice x}$ are supported on Cantor sets $\Lambda \cap \ell_x$, which we reasonably expect to be of Hausdorff dimension strictly between zero and one. 

Our interest is in the slice measure on the singular set $\condmeas := \rho_0$. In Figure~\ref{fig:CondMeasure}, a histogram of $\condmeas$ for the Lozi map at standard parameters $a = 1.7, b= 0.5$ is plotted.

\begin{figure}
	\centering
	\includegraphics{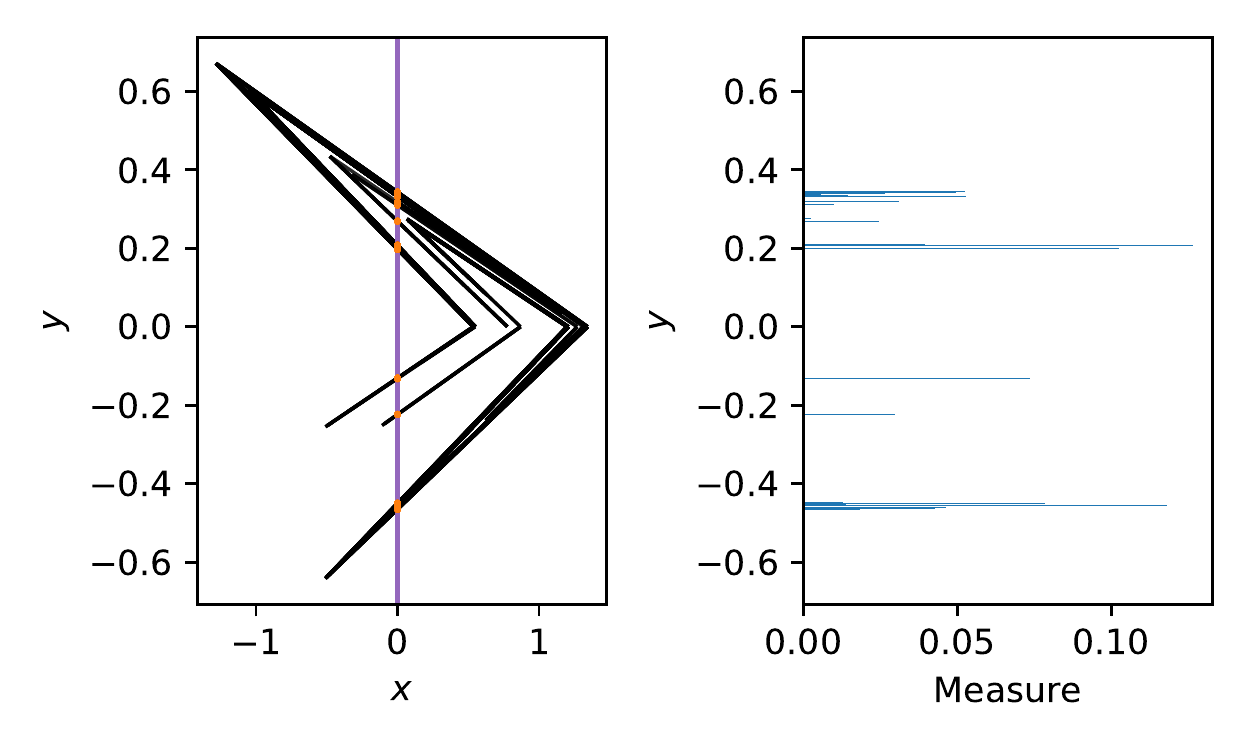}
	\caption{Left: picture of the Lozi attractor at $a=1.7, b=0.5$ (black), the singular line $\ell_{\S}$ (purple), their intersection (orange). Right: histogram of $\condmeas$ on $\ell_{\S}$, obtained from $200,\!000$ iterates of the unstable manifold dynamics $\vec f$, binned at width $0.0025$. Figure reprinted from \cite{mix}.}
	\label{fig:CondMeasure}
\end{figure}
%In \ref{mix} we made the following conjecture, with rigorously validated numerical evidence:
%\begin{conjecture}\label{c:Reduced}
%	Suppose that the geometric conditions above hold, and let $\condmeas$ be the slice probability measure of $\rho$ on the (one-dimensional) critical set. Then, there exist $C > 0$, $c \in (0,1)$ such that for all $A, B \in C^\Aord$ and $n \geq 0$,
%	\begin{equation} \left|\int_{\ell_{\S}} (A \circ f^{n})\,B\, \dd \condmeas - \rho(A) \condmeas(B) \right| \leq C \|A\|_{C^\Aord} \|B\|_{C^\Aord} c^{n}. \label{eq:ReducedConjecture} \end{equation}
%\end{conjecture}

We will make the following ``conditional mixing'' conjecture, %which predicts, with quantitative bounds, that $\condmeas$ converges weakly to $\rho$ when pushed forward by $f$ over time (i.e. as $n\to\infty$). 
which is analogous to \cite[Conjecture~4.1%\ref{mix-c:Reduced}
]{mix} and numerically supported in \cite{mix}:
\begin{conjecture}\label{c:Full}
	There exists a Banach space $\B$ with the following properties:
	\begin{itemize}
		\item $C^1(\M) \subseteq \mathcal{B}\subseteq L^\infty(\M)$;
		\item For any $Y \in C^2$ the functions $l^s Y,\, l^u Y \in \B$;
		\item There exist $c, \theta \in (0,1)$ and $C > 0$ such that for all $A,B, \Gamma$ and $m, n \geq 0$,
	\begin{equation} \left|%\int_{\ell_{\S}} 
	    \condmeas\left((A \circ f^{n+m})\, (B \circ f^m)\, \Gamma\, \lambda_m^{-1} \right) %\,\dd \condmeas 
		- \rho(A) \condmeas(B \circ f^m \Gamma \lambda_m^{-1}) \right| \leq C \|A\|_{C^\Aord} \|B\|_{\mathcal{B}} \| \Gamma \|_{BV(\S)} c^{n} \theta^m \label{eq:Conjecture1} \end{equation}
	and 
	\begin{equation} \left|%\int_{\ell_{\S}} 
		\condmeas\left( (A \circ f^{n})\, (B \circ f^{-m})\, \Gamma\, \mu_m\right)%\, \dd \condmeas 
		- \rho(A) \condmeas(B \circ f^{-m}\, \Gamma\, \mu_m) \right| \leq C \|A\|_{C^\Aord} \|B\|_{\mathcal{B}} \| \Gamma \|_{BV(\S)} c^n \theta^m,  \label{eq:Conjecture2} \end{equation}
	where $\mu(\psi) := \int_{\R^2} \psi\,\dd\mu$ for measures $\mu$ and functions $\psi$.
	\end{itemize}
\end{conjecture}
In other words, $A$ and the rest of the integral decorrelate as $n \to \infty$ in a way that is uniformly bounded (and in fact decaying exponentially in $m$, as expected given the presences of $\lambda_m^{-1}, \mu_m$). If $\condmeas$ lay, as $\rho$ does, in the Banach space defined in Section \ref{s:Sobolev} on which the transfer operator has a spectral gap, we would expect this result to hold.

Note that by setting $B=\Gamma=1$, $m=0$, we recover that for all $A\in C^\Aord$, $n\in\N$,
\begin{equation} \left|\int_{\ell_{\S}} (A \circ f^{n})\, \dd \condmeas - \rho(A) \condmeas(1) \right| \leq C \|A\|_{C^\Aord} c^{n}. \label{eq:ReducedConjecture} \end{equation}
for some $C$, $c\in(0,1)$, as expected in \cite[Conjecture~4.1%\ref{mix-c:Reduced}
]{mix}.
\\
Now, formally we expect a Lozi map $f$ to obtain a linear response to perturbations given by \eqref{eq:Perturbation} if the susceptibility coefficients
\begin{equation}
\kappa_n := \int_\M \nabla (A\circ f^n) \cdot X\, \dd\rho \label{eq:Susceptibility}
\end{equation}
are summable. Our main theorem is that Conjecture~\ref{c:Full} delivers this to us:
\begin{theorem}\label{t:Formal}
	Under Conjecture~\ref{c:Full}, there exist $C$, $c<1$ such that for all $A \in C^\Aord$, $X \in C^2$,
	\begin{equation}
		\left| \kappa_n \right| \leq C c^n \| A \|_{C^\Aord} \| X \|_{C^2}
	\end{equation}
	for some $C > 0$, $c \in (0,1)$.
\end{theorem}

\begin{figure}[htb]
	\centering
	\includegraphics{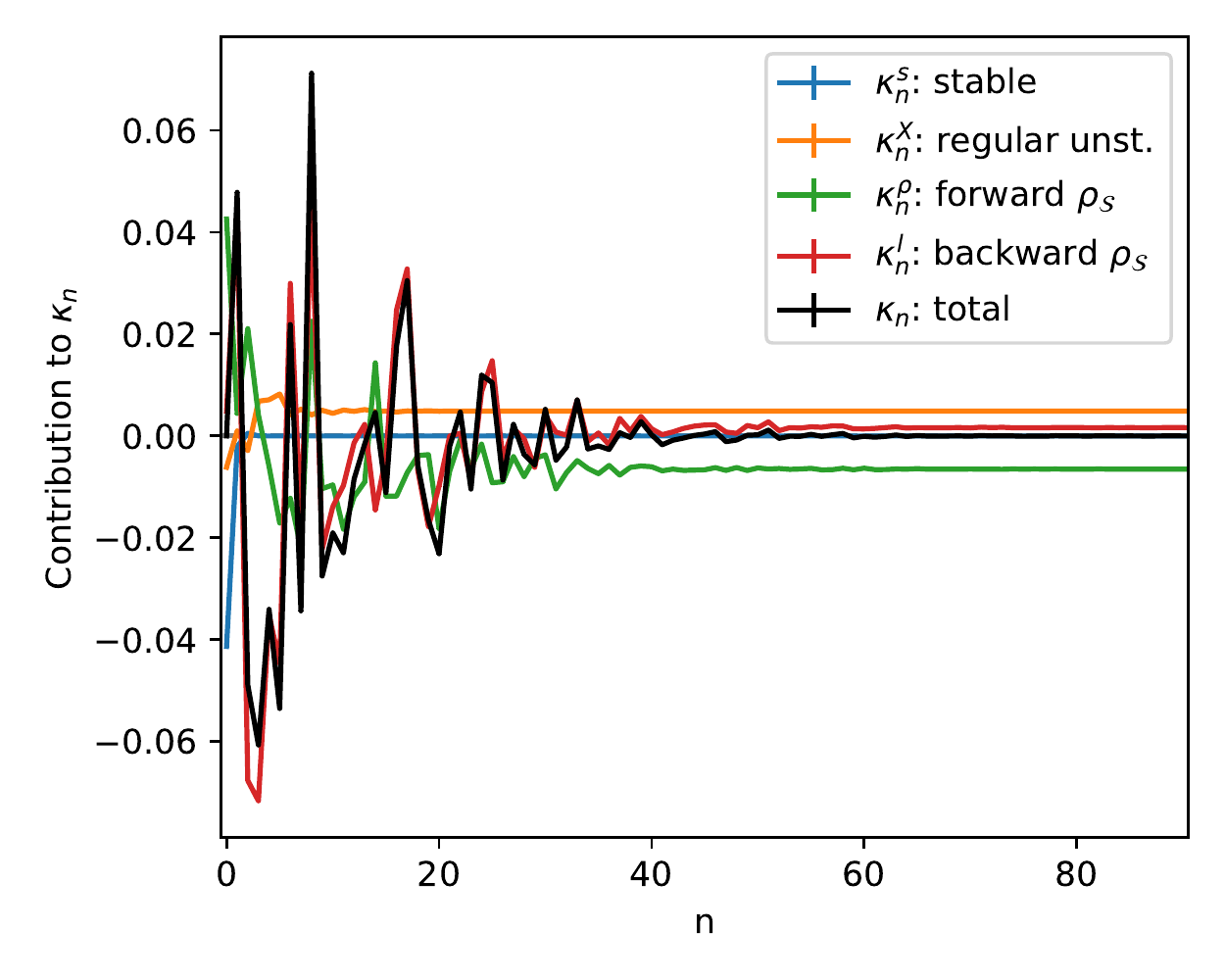}
	\caption{For the parameters of the Lozi map $a = 1.8$ and $b = 0.35$, the susceptibility function for the perturbation $X(x,y) = (0,y)$, and its components as given in Proposition~\ref{p:Decomposition}, using methods developed in \cite{mix}. Quadruple precision (106-bit mantissa) floating-point was used, with $100$ runs of length $400,000$. Error bars (too small to see) quantify the sampling error \cite{Gottwald16}.
	}
	\label{fig:Susceptibility}
\end{figure}

We might ask ourselves whether the existence of a formal linear response for the Lozi map can be supported numerically. To this end, we computed a Monte Carlo estimate of the susceptibility function for a Lozi map using the numerical implementation of the segment dynamics developed in \cite{mix} (our proof of Theorem~\ref{t:Formal}) will rely on similar notions) and the expressions in Proposition~\ref{p:Decomposition}. This estimate is plotted in Figure~\ref{fig:Susceptibility}. It can be seen that all terms in the susceptibility function exhibit exponential decay to their respective limits, which sum up to zero, as required.

We also present numerical evidence that this {\it formal} linear response actually translates to differentiability of the true response. In Figure~\ref{fig:Response} we plot the true and (formal) linear response for a family of Lozi maps $f_\varepsilon$ with parameters $(a,b) = (1.8,0.35(1+\varepsilon))$; in Figure~\ref{fig:ResponseDeviation} we plot the difference between these on a small parameter range. The formal linear response we obtain clearly appears to correspond to the true derivative of the response.

One might go a step further and ask how smooth the true response should in fact be. When studying systems with homoclinic tangencies, we suggested that Ruelle's argument could extend to saying that the response should be approximately $C^{1/2 + d_s}$ where $d_s$ is the stable dimension (i.e. the dimension of the attractor restricted to a stable manifold): we can further analogise this to say that piecewise hyperbolic systems should have an approximately $C^{1+d_s}$ response. We plot in green a possible Taylor approximation error based on this level of smoothness in Figures \ref{fig:Response}--\ref{fig:ResponseDeviation}: the error appears to be of the correct magnitude.
\begin{figure}
	\centering
	\includegraphics{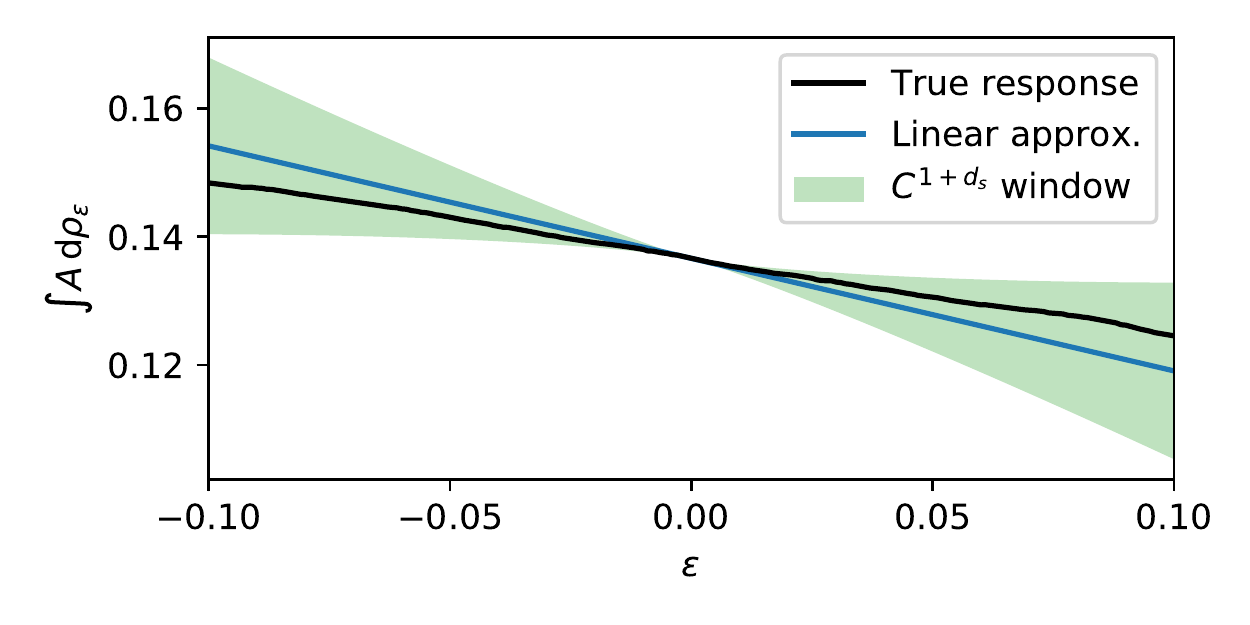}
	\caption{The true and linear response for the same parameters as in Figure~\ref{fig:Susceptibility}. The true response was obtained via Monte Carlo sampling; error bars (too small to see) quantify the sampling error. An error bound between the two of $0.15 |\varepsilon|^{1+d_s}$ is plotted in green, where we estimate $d_s = 0.26$ using the Kaplan-Yorke dimension.
	}
	\label{fig:Response}
\end{figure}

\begin{figure}
	\centering
	\includegraphics{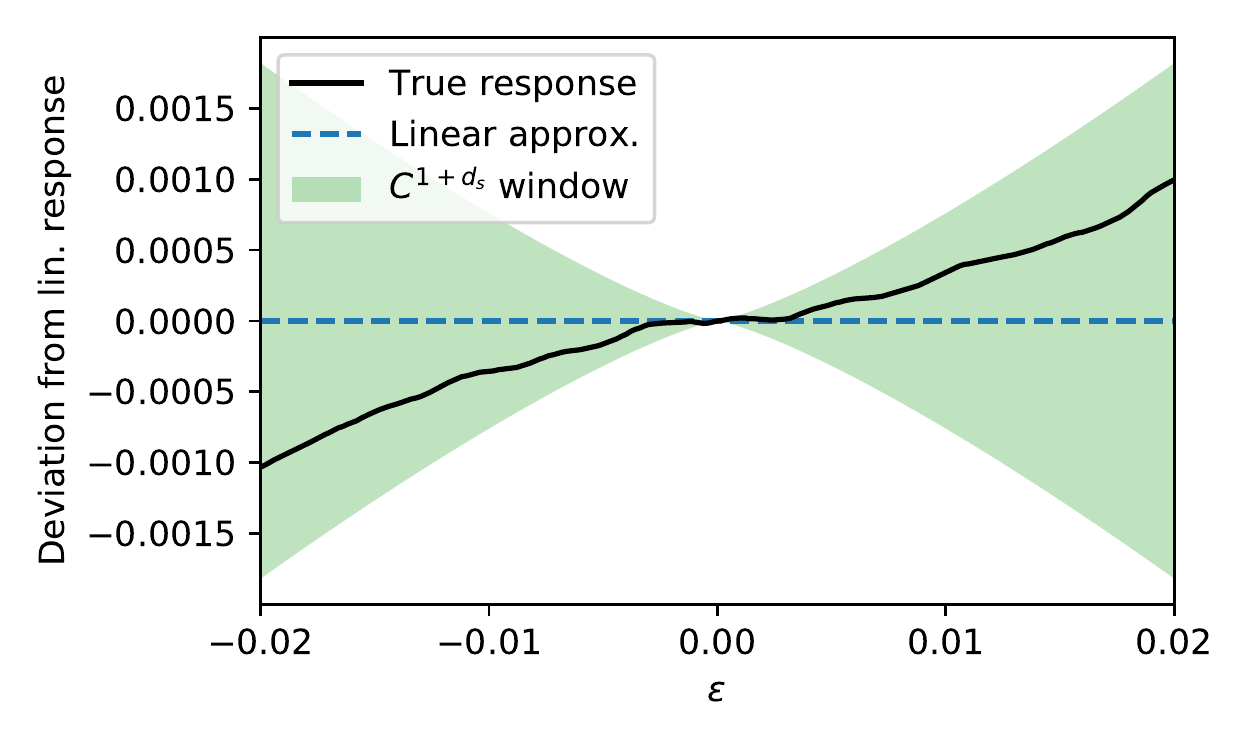}
	\caption{The difference between true and linear response for the same parameters as in Figure~\ref{fig:Susceptibility}. 
	}
	\label{fig:ResponseDeviation}
\end{figure}

Analogising to systems with stable-unstable tangencies, this corroborates the claim that ``generic smooth chaotic systems have $C^{d_s + 1/2}$ response'' \cite{Ruelle18}. Indeed, in the smooth case, exponential conditional mixing (perhaps in a more general, e.g. fractionally differentiated, form) is the mechanism required to obtain this.
%In fact, this result constitutes the strongest evidence of any kind yet for Ruelle's claim, as his argument, which is based on exponential mixing of standard smooth functions, in fact only gives a (formal) linear response for systems with homoclinic tangencies and $d_s > 3/2$\red{footnote why} (according with a $C^{-1/2 + d_s}$ response for these systems). On the other hand, we have evidence that SRB measures conditioned on singularity sets mix exponentially, which implies a formal linear response for $d_s>1/2$. The phenomenon of cross-sectional mixing therefore fills in a piece of the argument that ``general smooth chaotic systems have $C^{d_s + 1/2}$ response''.
\\

The remainder of this paper will be largely be tasked with proving Theorem~\ref{t:Formal}. Let us sketch the idea of the proof. Because the Lozi map is uniformly hyperbolic, the perturbing vector field $X$ can be split into stable and unstable parts. The susceptibility coefficients $\kappa_n$ can therefore too be split into stable and unstable parts (Proposition~\ref{p:Decomposition}). It is standard that the stable part $\kappa^s_n$ is exponentially decaying (Proposition~\ref{p:KappaS}); as with smooth hyperbolic systems the unstable part contains a correlation term $\kappa^X_n$ between $A$ and a smooth function, which is therefore exponentially decaying (Proposition~\ref{p:KappaX}). However, there are two extra terms: one, $\kappa^\rho_n$, arises from discontinuities in the unstable vector field (Proposition~\ref{p:KappaRhoSum}), and the other, $\kappa^l_n$, arises from discontinuities in the stable vector field (i.e. unstable covector field, Proposition~\ref{p:KappaLSum}). These discontinuities are generated by the discontinuity of the map's derivative and so lie in the forward (respectively backward) orbit of the singularity set, hence the application of Conjecture~\ref{c:Full}.

\section{Decay of correlations for regular distributions}\label{s:Sobolev}

To prove our results, we will need to introduce some theoretical tools. In this section, we present the current exponential decay of correlations theory.

As an application of \cite[Theorem~2.5]{BaladiGouezel10} we have exponential decay of correlations for the Lozi map in a certain Banach space:
\begin{proposition}\label{p:DecayOfCorrelations}
	 There exists a Banach space of distributions $(C^1)^* \supset \Ht \supset C^1$, given in \cite[Theorem~2.5]{BaladiGouezel10} and $C>0$, $\xi \in (0,1)$ such that for all functions $A \in C^1$ and all finite signed Borel measures $\phi \in \Ht$, we have
	\[ \left|\ \int_\M A\circ f^n\, \dd\phi - \int_\M A\, \dd \rho \int_\M \dd\phi\ \right| \leq C \|A\|_{C^1} \|\phi\|_{\Ht} \xi^n. \]
\end{proposition}
We prove this result in the Appendix. 
	
We will not interact with the construction of $\Ht$ directly, but will instead use some more abstract results to construct functions lying in $\Ht$. The first result is as follows:
\begin{lemma}[\cite{BaladiGouezel10}, Lemma~4.1]\label{l:FunctionMultipliers}
	There exists a constant $C_\flat$ such that for any function $\phi \in \Ht$ and any function $g \in C^1$,
	\[ \| g \phi \|_{\Ht} \leq C_\flat \| g \|_{C^1} \| \phi \|_{\Ht}. \]
\end{lemma}

The second result regards multiplication by dynamically relevant characteristic functions. If the symbol space of the Lozi map is $\Sigma = \{+,-\}$ corresponding respectively to the sets $\M_{\pm}$, we can define cylinders at the beginning and at the end for $n \in \N$ and $\cyi \in \Sigma^n$:
\begin{align} \O^b_\cyi &= \bigcup_{m=1}^{n} f^{1-m}(\M_{\cyi_m}) \\
\O^e_\cyi &= f^n(\O^b_\cyi) = \bigcup_{m=1}^{n} f^{n+1-m}(\M_{\cyi_m}). \label{eq:FinalCylinderDef}
\end{align}
The following lemma says that multiplication by the characteristic functions of these sets is nice:
\begin{lemma}\label{l:CylinderMultipliers}
	There exists a constant $C_\#$ such that for any function $\phi \in \Ht$ and any initial or final cylinder $\O$,
	
	\[ \| \mathbb{1}_{\O} \phi \|_{\Ht} \leq C_\# \| \phi \|_{\Ht}. \]
\end{lemma}

Lemma~\ref{l:CylinderMultipliers} requires us to understand the geometry of the cylinders, in particular the number of intersections they have with lines in the stable cone. A more general but relatively weak bound is given in the proof of \cite[Lemma~5.1]{BaladiGouezel10}; for simplicity in the sequel we will use a strong bound that arises from the following property of Lozi maps:
\begin{proposition}\label{p:CylinderIntersections}
	Intersections between any initial or final cylinders $\O^b_{\cyi}, \O^e_{\cyi}$ and any line in $\R^2$ are connected.
\end{proposition}
\begin{proof}[Proof of Proposition~\ref{p:CylinderIntersections}]
	We can prove by induction on the length of $\cyi$ that the final cylinders are convex. The cylinders of indices of length $1$ are $\O^e_{[\pm]} = \M_{\pm}$ which are convex; from \eqref{eq:FinalCylinderDef} and since $f$ is a bijection we have that
	\[ \O^e_{\cyi} = f(\O^e_{\cyi_{1:n-1}}\cap \M_{\cyi_n}).\]
	Now, the intersection of two convex sets is convex; $f$ applied to $\O^e_{\cyi_{1:n-1}}\cap \M_{\cyi_n} \subseteq \M_{\cyi_n}$ is linear, so preserves convexity, giving us the inductive step.
	
	By definition we then have that for $k \leq n$, $f^{-k}(\O^e_{\cyi}) \subseteq \M_{\cyi_{n+1-k}}$, and so $f^{-n}$ applied to $\O^e_{\cyi}$ is linear. Thus, the initial cylinders $\O^b_{\cyi} = f^{-n}(\O^e_{\cyi})$ are also convex.
	
	Lines are also convex objects, and thus their intersections with the cylinders are also convex and therefore connected.
\end{proof}

\begin{proof}[Proof of Lemma~\ref{l:CylinderMultipliers}]
	Our cylinders are convex so their intersections with any line (in particular, any line in the stable cone) will have exactly one connected component: an application of \cite[Lemma~4.2]{BaladiGouezel10} combined with \cite[Definition~2.12]{BaladiGouezel10} gives the boundedness of this multiplication.
\end{proof}
	
\section{Disintegration and measures}\label{s:Disintegration}

To tame the susceptibility function \eqref{eq:Susceptibility} we will need to perform an integration by parts, requiring us to differentiate the SRB measure in some sense. This necessitates an understanding of the structure of the SRB measure, which we develop in this section; in the course of this we will prove Theorem~\ref{t:SliceMeasures}.

\subsection{Unstable manifold dynamics}

To understand the structure of the SRB measure we will find it useful to lift the dynamics on points onto dynamics on local unstable manifolds. To do this, we must first define these objects.

Every point in $\Lambda \backslash \bigcup_{n=1}^\infty f^n(\S)$ has a local unstable manifold
\begin{equation} \Wul(x) = \left\{ y \in \Lambda : \lim_{n\to\infty} \| f^{-n}(y) - f^{-n}(x) \| = 0,\, \forall n \geq 0\ \M_{f^{-n}(y)} = \M_{f^{-n}(x)} \right\}, \label{eq:LocalUnstableManifold} \end{equation}
where we let $\M_{(p_1,p_2)} := \M_{\sign p_1}$ be the domain in which the point $(p_1,p_2)$ lies.
It follows naturally that $x \sim y \iff y \in \Wul(x)$ is an equivalence relation.

Let
\begin{equation} I_{p,q} := \{ (1-t) p + t q : t \in (0,1) \}. \label{eq:IpqDefinition} \end{equation}
denote an segment between $p$ and $q$. Conversely, let $p_I$ and $q_I$ respectively denote the start-point and end-point of an segment $I$, where either $I$ has been given a direction or the choice does not matter.

%\begin{proposition}\label{p:H4}
%	(H4) in \cite{Young85} holds: in particular, there exist $k \in \N$ such that $f^j(\S) \cap \S$ is empty for all $j = 1,\ldots, k$, and
%	\[ \inf_{x \in \M, v \neq 0} \frac{\| \left(\begin{smallmatrix}1 & 0 \\ 0 & b^{-1}\end{smallmatrix}\right)D_x f v\|}{\| \left(\begin{smallmatrix}
%		1 & 0 \\ 0 & b^{-1}\end{smallmatrix}\right)v\|} > 2.\]
%\end{proposition}
%We find that (H4) holds with $k=2$ for $0 < b < a - \sqrt{2}$. %Fundamentally, Assumption~\ref{a:H4} serves to bound the complexity of the map, thus ensuring it has various nice properties: for us, that local unstable manifolds have positive length (Proposition~\ref{p:LUMIsSegment}), and that certain conditional measures are finite (Proposition~\ref{p:SigmaFinite}).

Since the Lozi map is piecewise affine, it has certain pleasant affine properties \cite{ColletLevy84}:
\begin{proposition}\label{p:LUMIsSegment}
For $\rho$-almost every $x \in \Lambda$, there exist distinct $p,q \in \Lambda$ such that $\Wul(x) = I_{p,q}$. Furthermore, the conditional measure of $\rho$ on $\Wul$ is the uniform measure.
\end{proposition}
%\begin{proof}
%	By \cite{Young85}\footnote{Note that it is necessary to check that the assumption (H4) in this paper holds for the Lozi map parameters we are interested in: a calculation shows that this is the case, with $k=2$.}, $\rho$-a.e. $x \in \Lambda$ has a local unstable manifold $\Wul(x)$ of positive length. Because each preimage $f^{-n}(\Wul(x))$ always lies entirely either in $\M_+$ or $\M_-$, each $f^{-n} |_{\Wul(x)}$ is an affine map and so the unstable direction is constant on $\Wul(x)$. Thus $\Wul(x)$ is a straight line.
%\end{proof}

Let us therefore define the set of directed local unstable manifolds
\[ \vec \L = \{ \vec I : \exists x \in \Lambda\ I = \Wul(x) \} \]
where $\vec I$ is a {\it directed} segment in $\R^2$ (that is, start points and end points are distinguished). Directedness of manifolds will become useful to us when in future we wish to take directional derivatives.

%Defining the orientation reversing involution
%\[ \mathfrak{r}(\vec I_{p,q},t) = (\vec I_{q,p}, 1-t), \] 
%and the equivalence relation $\mathfrak{r}(w) \sim w$, we can quotient out the orientability to obtain
On the other hand, we can define a set of undirected local unstable manifolds
\begin{align*}
	\hat \L &= \{\Wul(x): x \in \Lambda\}
%	\hat \Lambda &= \vec \Lambda / \mathfrak{r}.
\end{align*}
which is in the obvious two-to-one relationship with $\vec \L$.

Let us also define the following product space, which we will use to parametrise each $\vec I \in \vec \L$
\[ \vec \Lambda = \vec \L\times (0,1).\]
Thus, $\vec \Lambda$ can be understood as containing the set of points in $\Lambda$ with their (directed) unstable manifolds, through the $\rho$-almost everywhere two-to-one map $\pi: \vec \Lambda \to \Lambda$
\[ \pi(\vec{I},t) := (1-t) p_{\vec{I}} + t q_{\vec{I}}, \]
where we denote the start point (resp. end point) of the directed segment $\vec I$ to be $p_{\vec I}$ (resp. $q_{\vec I}$).
%From (\ref{eq:LUMIsSegment}) and the equivalence class property of maps, we can write almost every point $x \in \Lambda$ uniquely as
%\[ x = (1-t) p + t q \]

The collection of undirected segments $\hat \L$ naturally inherits a measure $\hat\rho$ from $\rho$ via $\dd\hat\rho(I) = \dd\rho(I)$. This $\hat\rho$ is known as the transverse measure in the disintegration of $\rho$ into unstable manifolds. Furthermore, using Proposition~\ref{p:LUMIsSegment},
\begin{equation}
\int_\M A(x)\, \dd\rho(x) = \int_{\hat \L} \int_0^1 A(\pi(\vec I,t))\, \dd t\, \dd\hat\rho(I) \label{eq:Disintegration}
\end{equation}
for arbitrary choice of orientation $\vec I$.
\\

Our local unstable manifolds are covariant under the flow, except that when they cross the critical line $\ell_{\S}$ they are cut into two pieces. In particular, for $x \notin \ell_{\S}$,
\[ \Wul(f(x)) = f(\Wul(x) \cap \M_x), \]
Let us define the (one-step) {\it descendants} of an segment $\vec I \in \vec \L$ to be the appropriately oriented local unstable manifolds contained in $f(\vec I)$:
\begin{equation} \D(\vec I) := \{ f(\M_+ \cap \vec I), f(\M_- \cap \vec I)\}, \label{eq:Children}\end{equation}
and its non-oriented equivalent $\D(I)$.
We have that 
\[ \overline{f(\Wul(x))} = \bigcup_{y \in \Wul(x)} \overline{\Wul(f(y))}, = \overline{\cup \D(\Wul(x))},\]
a union of at most two distinct segments.

\subsection{Measures on the singular line}

%Supposing that local unstable manifold $\Wul(x)$ is a proper segment (i.e. contains more than one point), and assigning a direction to $\Wul(x)$, we can identify $x$ with its unstable manifold combined with a parameter $t \in [0,1]$ saying where it is on the manifold:

Let us define the singular (non-probability) measure for $E \subseteq \S$ as
\begin{equation} \hat{\sigma} (E) = \int_{\hat \L} |I|^{-1}\,\dd\hat{\rho}(I), \label{eq:SigmaDefinition}\end{equation}
%recalling that we defined segments $I$ so as to not contain their endpoints.
 We also have the vector equivalent:
\begin{equation} \vec{\sigma} (E) := \int_{\vec \L} %\cha(\vec I \cap E \neq \emptyset)
\,\frac{\dd\hat{\rho}(\vec I)}{2 (q-p) \cdot \vu(\vec I)}, \label{eq:VecSigmaDefinition}\end{equation}
with $\dd|\vec \sigma|([\vec I]) = \hat\sigma([I])$ and $\dd\vec\sigma(\vec I_{p,q}) = -\dd\vec\sigma(\vec I_{q,p})$. We introduce the factor of $2$ to allow both orientations for the segments $\vec I$.

The measures $\vec{\sigma}$ and $\hat{\sigma}$ are closely related to the ``slice'' measures $\rho_{\slice x}$, which we will now define. For all Borel sets $E \subset \ell_x$, we define $\rho_{\slice x}(\R^2 \backslash E) = 0$ and
	\begin{equation} \rho_{\slice x}(E) = \int_{\hat \L} \left(\sum_{s \in I \cap E} (\vu \cdot e_1)(s)^{-1} + \tfrac{1}{2} \sum_{s \in \{p_I, q_I\} \cap E} (\vu \cdot e_1)(s)^{-1}\right)\,\dd\hat\sigma(I). \label{eq:RhoMidDef}\end{equation}
	
	The following lemma gives us, pleasantly, that the second summand can be omitted for $\rho_{\slice 0} = \condmeas$:
	\begin{lemma}\label{l:CondMeasSigma}
		For all $A \in L^1(\condmeas)$, the following relation holds:
		\[ \int_{\hat{\L}} \sum_{s \in I \cap \ell_{\S}} A(s)\, \dd \hat\sigma(I) = \int_{\ell_{\S}} A(s) \vu(s)\cdot e_1\, \dd\condmeas(s). \]
		Furthermore, $0 < c \leq \vu \cdot e_1 \leq 1$ for some $c$.
	\end{lemma}
We will use the following proposition, proved in the appendix, to show this:
\begin{proposition}\label{p:CriticalLineSelfIntersectionMeasure}
	The forward orbit of $\condmeas$-almost every $y \in \ell_{\S}$ has no other intersections with $\ell_{\S}$.
\end{proposition}
%This result, which is proven in the appendix, implies the same result for the backward orbit.
	\begin{proof}[Proof of Lemma~\ref{l:CondMeasSigma}]
		From \eqref{eq:RhoMidDef}, integrating the measurable function $(\vu\circ e_1) A$ over $\condmeas = \rho_{\slice 0}$ gives
		\begin{equation} \int_{\ell_{\S}} A(s) \vu(s)\cdot e_1\, \dd\condmeas(s) = \int_{\hat \L} \left(\sum_{s \in I \cap \ell_{\S}} A(s) + \tfrac{1}{2} \sum_{s \in \{p_I, q_I\} \cap \ell_{\S}} A(s)\right)\,\dd\hat\sigma(I). \label{eq:FullCondIntegral}\end{equation}
		Now, Proposition~\ref{p:CriticalLineSelfIntersectionMeasure} states that the set of points $\{s \in \ell_{\S} \mid \exists \vec I\in\vec{\L}\ s \in \{p_{\vec I}, q_{\vec I}\}\}$ has zero $\condmeas$-measure. Therefore, from \eqref{eq:RhoMidDef} the set of intervals containing these points---i.e. where the second sum in \eqref{eq:FullCondIntegral} contributes---must also have zero $\hat{\sigma}$ measure. 
%		Applying \eqref{eq:RhoMidDef} to  and then that $\
%		By Proposition~\ref{p:CriticalLineSelfIntersectionMeasure} and
	\end{proof}
	
	We now begin to unwind the relationship between the $\rho_{\slice x}$ and $\hat{\sigma}$, beginning by proving finiteness of the $\rho_{\slice x}$.
	
	\begin{proof}[Proof of Theorem~\ref{t:SliceMeasures}\ref{res:finite}]
	Define the slice $L_{x,\delta} = (-x-\delta,x+\delta)\times\R$,
	Considering the disintegration to unstable segments, for either choice of orientation $\vec I$ of $I$, we find% and any continuous function $A: \R^2 \to \R$,
	\begin{align}
	\frac{1}{2\delta} \int_{0}^1 %A(\pi(\vec I,t)) 
	\mathbb{1}(\pi(\vec I,t) \in L_{x,\delta}) \dd t &= \frac{1}{2\delta} %A(\pi(\vec I,t)) 
	\int_{0}^1 \mathbb{1}\left((1-t)p_{\vec I} + tq_{\vec I} \in L_{x,\delta}\right)\,\dd t\notag\\
	&= \frac{\Leb\left([x_{p_{\vec I}},x_{q_{\vec I}}] \cap (x-\delta,x+\delta)\right)}{2\delta |x_{q_{\vec I}} - x_{p_{\vec I}}|}
	%(A(s) + o_{A}(\delta))
	. \label{eq:MonotoneThing}
	\end{align}
	Now, for any choice of $s \in I$, $|x_{q_{\vec I}} - x_{p_{\vec I}}| = \vu(s)\cdot e_1 |I|$. As $\delta \to 0$ we therefore recover the limit
	\begin{equation}
	\lim_{\delta \to 0} \frac{1}{2\delta} \int_{0}^1 \mathbb{1}(\pi(\vec I,t) \in L_{x,\delta})\, \dd t = |I|^{-1} \left(\sum_{s \in I \cap \ell_x} (\vu \cdot e_1)(s)^{-1} + \tfrac{1}{2} \sum_{s \in \{p_I, q_I\}} (\vu \cdot e_1)(s)^{-1}\right). \label{eq:SliceToSigma}
	\end{equation}
	
	On the other hand, considering the measure as a whole, we have that
	\[ \frac{1}{2\delta} \int_{L_{x,\delta}} \dd\rho = \frac{1}{2\delta} \int_{x-\delta}^{x+\delta} \dd \pi_x^*\rho, \]
	where $\pi_x$ is the coordinate in the $x$ direction. By \cite[Lemma]{Young85}, the measure $\pi_x^*\rho$ is absolutely continuous with bounded density, and so in particular, there exists a constant $C$ such that for all $\delta > 0$, $x \in \R$,
	\begin{equation} \frac{1}{2\delta} \int_{L_{x,\delta}} \dd\rho \leq C. \label{eq:FiniteSliceBounded}\end{equation}
	We can of course disintegrate by the unstable measure to say that 
	\begin{equation}
	\int_{\vec{\L}} \frac{1}{2\delta} \int_{0}^1 %A(\pi(\vec I,t)) 
	\mathbb{1}(\pi(\vec I,t) \in L_{x,\delta}) \dd t\,\dd\vec\rho(\vec I) \dd\rho \leq C \label{eq:SliceMeasureBounded}
	\end{equation}
	
	Let us attempt to combine these. Let us define $\vec{L}_x = \{\vec I \in \L: \bar{\vec I} \cap \ell_x \neq 0\}$ to be the set of directed segments that intersect, or whose endpoints coincide with, the line $\ell_x$. Given this is a subset of $\vec{L}$, we can use \eqref{eq:SliceMeasureBounded} to say that
	\[ \int_{\vec{\L}_x} \frac{1}{2\delta} \int_{0}^1 %A(\pi(\vec I,t)) 
	\mathbb{1}(\pi(\vec I,t) \in L_{x,\delta}) \dd t\,\dd\vec\rho(\vec I) \dd\rho \leq C. \]
	Now, right-hand side is an integral of \eqref{eq:MonotoneThing} over different segments $\vec I \in \vec{L}_x$. It is easy enough to see that when $\vec I$ intersects $\ell_x$, the expression in \eqref{eq:MonotoneThing} is increasing in $\delta$, so we can use the monotone convergence theorem to say that
	\begin{align*} C &\geq \int_{\vec{\L}_x} \frac{1}{2\delta} \int_{0}^1 %A(\pi(\vec I,t)) 
	\mathbb{1}(\pi(\vec I,t) \in L_{x,\delta}) \dd t\,\dd\vec\rho(\vec I)\\
	&= \int_{\vec{\L}_x} \left(\sum_{s \in I \cap \ell_x} (\vu \cdot e_1)(s)^{-1} + \tfrac{1}{2} \sum_{s \in \{p_I, q_I\}} (\vu \cdot e_1)(s)^{-1}\right) |I|^{-1}\,\dd\vec\rho(\vec I)
	\end{align*}
	using \eqref{eq:SliceToSigma} in the last equality. We can extend the domain of integration of this integral from $\vec{L}_x$ to $\vec{L}$ without changing its value. We then have that $|I|^{-1}\,\dd\vec\rho(\vec I) = |\dd\vec\sigma(\vec I)|$, and can drop the directionality. This just recovers $\rho_{\slice x}(\ell_x)$ from \eqref{eq:RhoMidDef}, and so we have a uniform bound on $\rho_{\slice x}(\R^2)$ as required for part \ref{res:finite}.
	%	Now, we can decompose the expression inside the limit in \eqref{eq:SliceMeasureProperty} as
	%	\[ \frac{1}{2\delta} \int_{(x-\delta,x+\delta) \times \R} A\,\dd\rho = \frac{1}{2\delta} \int_{\vec\L} \mathbb{1}(\pi(\vec I,t) \in (-x-\delta,x+\delta)\times\R) A(\pi(\vec I,t))\,\dd t \,\dd\vec\rho =  \]
\end{proof}

The finiteness of $\hat{\sigma}$ (therefore of $\vec{\sigma}$) will follow from that of $\rho_0 = \condmeas$:
\begin{proposition}\label{p:SigmaFinite}
 	$\hat{\sigma}$ is a finite measure.
\end{proposition}

To prove this, we will need a lemma. The following result will also be useful in understanding the behaviour of integrals with respect to $\vec\sigma$ under $f$ dynamics. This will enable us to prove Proposition~\ref{p:SigmaFinite}, as well as iteratively reducing various integrals over forward/backwards orbits of $\ell_{\S}$ to integrals over $\ell_{\S}$.
\begin{lemma}\label{l:SigmaForwardMap}
	For any $\psi: \vec \L \to \R$, and any choices of point $u_{\vec J} \in f^{-1}(\vec J)$ for each $\vec J \in \vec \L$,
	\[ \int_{\vec \L} \psi(\vec I) \,\dd \vec \sigma(\vec I) = \int_{\vec \L} \sum_{\vec J \in \D(\vec I)} \lambda_1(u_{\vec J})^{-1} \psi(\vec J) \,\dd \vec \sigma(I). \]
\end{lemma}
\begin{proof}
	Recall from \eqref{eq:Children} that the descendants $\D(\vec I)$ of a directed segment $\vec I \in \vec\L$ are the (up to two) directed segments created by applying the map $f$ to it.
	
	Because each segment $\vec J \in \vec \L$ is the descendant of exactly one segment $\vec I \in \vec \L$, we can write
	\[ \int_{\vec \L} \psi(\vec I) \,\dd \vec \sigma(\vec I) = \int_{\vec \L} \sum_{\vec J \in \D(\vec I)} \psi(\vec J) \,\dd \vec \sigma(\vec J). \]
	Now from the definition of $\vec \sigma$ in \eqref{eq:VecSigmaDefinition},
	\[ 2\dd \vec \sigma(\vec J_{p',q'}) = \frac{\dd \rho(\vec J_{p',q'})}{(q' - p') \cdot e_1} = \frac{\dd \rho(f^{-1}(\vec J_{p',q'}))}{(q'-p') \cdot \vu(f(u_{\vec J_{p',q'}}))}. \]
	Recalling that $f^{-1}$ is affine on $\vec J_{p',q'}$ and the definition of $\lambda_n$ \eqref{eq:LambdaNDefinition}, 
	\[ (q-p) \cdot \vu(f(u_{\vec J_{p',q'}})) = \lambda_1(u_{\vec J_{p',q'}}) (f^{-1}(q')-f^{-1}(p')) \cdot \vu(f(u_{\vec I})). \]
	Because $\vec J$ is oriented in the same direction as $f(\vec I)$, we have that 
	\[ (f^{-1}(q')-f^{-1}(p')) \cdot \vu(u_{f(\vec I)}) = \frac{\| f^{-1}(q') - f^{-1}(p')\|}{\| q - p \| } (q - p) \cdot \vu(f(u_{\vec I})).\]
	Finally, because $\rho$ is uniform on unstable leaves, 
	\[  \frac{\| f^{-1}(q') - f^{-1}(p')\|}{\| q - p \| } = \frac{\dd \vec\rho(f^{-1}(\vec J))}{\dd \vec \rho(\vec I)}, \]
	and hence we find
	\[ \int_{\vec \L} \psi(\vec I) \,\dd \vec \sigma(\vec I) = \int_{\vec \L} \sum_{\vec J \in \D(\vec I)} \psi(\vec J) \,\frac{\dd \vec \rho(\vec I)}{\lambda_1(u_{\vec J})\, 2(q - p) \cdot \vu(f(u_{\vec I}))}, \]
	which through the definition of $\vec \sigma$ gives the required result.
\end{proof}

\begin{proof}[Proof of Proposition~\ref{p:SigmaFinite}]
	We are given that $\condmeas(\ell_{\S}) < \infty$ by Theorem~\ref{t:SliceMeasures}\ref{res:finite}. If we set $A \equiv 1$ in Lemma~\ref{l:CondMeasSigma} then we have that 
	\begin{equation} \int_{\hat{\L}} |\{I \cap \ell_{\S} \neq \emptyset\}| \, \dd \hat\sigma(I) = \int_{\ell_{\S}} \dd \condmeas =: K < \infty. \label{eq:IntegralICapC} \end{equation}
	That is, the $\hat{\sigma}$-measure of segments intersecting the singular line is finite.
	
	Set for $n \in \N \cup \{\infty\}$
	\[ F_n = \bigcup_{k=1}^n \{I \in \hat{\L}: \D^{-k}(I) \cap \ell_{\S} \neq \emptyset\}. \]
	
	Now, by Lemma~\ref{l:SigmaForwardMap},
	\[ \int_{\hat{\L}} \cha_{I \in F_{n+1}} \dd \hat\sigma(I) = \int_{\hat{\L}} \sum_{\vec J \in \D(\vec I)} \lambda_1(u_{\vec J})^{-1} \cha_{J \in F_{n+1}}\,\dd\hat{\sigma}(I)\]
	Since each descendant $J$ has $\D^{-1}(J) = I$, this implies that
	\begin{equation} \int_{\hat{\L}} \cha_{I \in F_{n+1}} \dd \hat\sigma(I)  \leq \lambda^{-1} \int_{\hat{\L}} |\D(\vec I)| \cha_{I \in F_{n}} \dd \hat\sigma(I). \label{eq:IntegralChildrenExpansion}\end{equation}
	Now, the number of times $I$ is cut by the singularity line $|\{I \cap \ell_{\S} \neq \emptyset\}|$ will be one less than the number of descendants of $I$. Recalling additionally that characteristic functions are bounded by one and that $F_n \subseteq F_{n+1}$ we can split
	\begin{align*} \int_{\hat{\L}} |\D(\vec I)| \cha_{I \in F_{n}}\, \dd \hat\sigma(I) & = \int |\{I \cap \ell_{\S} \neq \emptyset\}|\, \dd \hat\sigma(I) + \int_{\hat{\L}} \cha_{I \in F_{n}}\, \dd \hat\sigma(I) \\
	&\leq \int |\{I \cap \ell_{\S} \neq \emptyset\}|\, \dd \hat\sigma(I) + \int_{\hat{\L}} \cha_{I \in F_{n+1}}\, \dd \hat\sigma(I) \end{align*}
	which combined with \eqref{eq:IntegralICapC} and \eqref{eq:IntegralChildrenExpansion} gives us that
	\[ \int_{\hat{\L}} \cha_{I \in F_{n+1}} \dd \hat\sigma(I) \leq \lambda^{-1}\left(K + \int_{\hat{\L}} \cha_{I \in F_{n+1}}\, \dd \hat\sigma(I) \right) \]
	and so
	\[ \int_{\hat{\L}} \cha_{I \in F_{n+1}} \dd \hat\sigma(I) \leq K/(\lambda - 1). \]
	Now, from \cite[Proposition~A.1%\ref{mix-p:CriticalOrbitSingularities}
	]{mix} we have that $F_{\infty}$ has full $\hat\rho$ (and thus $\hat\sigma$) measure, which means that by the monotone convergence theorem
	\[ \int_{\hat{\L}} \dd \hat\sigma(I) \leq K/(\lambda - 1) < \infty. \]
\end{proof}

We finish this section by proving the remainder of Theorem~\ref{t:SliceMeasures}.
\begin{proof}[Proof of Theorem~\ref{t:SliceMeasures}\ref{res:slice}--\ref{res:disintegration}]

To prove part \ref{res:slice}, we decompose
\[ \frac{1}{2\delta} \int_{L_{x,\delta}} A\,\dd\rho = \int_{\vec L} \frac{1}{2\delta} \int_0^1 A(\pi(\vec I,t)) \mathbb{1}(\pi(\vec I,t) \in L_{x,\delta})\,\dd t\,\dd \vec{\rho}(x). \]
By a similar argument used to bound \eqref{eq:MonotoneThing}, we have
\begin{align*} \frac{1}{2\delta} \int_0^1 A(\pi(\vec I,t)) \mathbb{1}(\pi(\vec I,t) \in L_{x,\delta})\,\dd t &\leq \frac{1}{2\delta} \|A\|_{L^\infty} \| (\vu\cdot e_1)^{-1}\|_{L^\infty} \min\{1,2\delta/|\vec I|\}\\
& \leq \frac{C \|A\|_{L^\infty}}{|\vec I|} \end{align*}
for some $C<\infty$, recalling that by the definition of the unstable cone $\C^u$, $\vu\cdot e_1$ is bounded away from zero. From Proposition~\ref{p:SigmaFinite}, $\vec{\sigma}$ must be finite (because $\hat{\sigma}$ is): this is equivalent to saying that $|\vec I|^{-1}$ is integrable with respect to $\vec{\rho}$. Hence, by the dominated convergence theorem, 
\begin{align*}\lim_{\delta\to 0} \frac{1}{2\delta} \int_{L_{x,\delta}} A\,\dd\rho = \int_{\vec L} \lim_{\delta \to 0} \frac{1}{2\delta} \int_0^1 A(\pi(\vec I,t)) \mathbb{1}(\pi(\vec I,t) \in L_{x,\delta})\,\dd t\,\dd\vec\rho.
\end{align*} 
As in part \ref{res:finite}, and using the continuity of $A$ along unstable leaves, we obtain that
\[ \lim_{\delta \to 0} \frac{1}{2\delta} \int_0^1 A(\pi(\vec I,t)) \mathbb{1}(\pi(\vec I,t) \in L_{x,\delta})\,\dd t = |I|^{-1}\left( \sum_{s \in \vec I \cap \ell_{\S}} \frac{A(s)}{(\vu\cdot e_1)(s)} + \tfrac{1}{2} \sum_{s\in \{p_{\vec I},q_{\vec I}\}\cap \ell_{\S}} \frac{A(s)}{(\vu\cdot e_1)(s)} \right) \]
and so by the definition of $\rho_{\slice x}$ in \eqref{eq:RhoMidDef},
\[ \lim_{\delta \to 0} \frac{1}{2\delta} \int_0^1 A(\pi(\vec I,t)) \mathbb{1}(\pi(\vec I,t) \in L_{x,\delta})\,\dd t = \int_{\ell_x} A\,\dd\rho_{\slice x}, \]
giving what is required.
\\

We now turn to part \ref{res:disintegration}. %If $E\subset \R^2$ is a Borel-measurable set, then the characteristic function $\mathbb{1}_E$ is measurable and bounded, and therefore also $\rho_{\slice x}$ integrable. 
%For any measurable $E \subset \R^2$ and $x \in \R$ we are given that $\rho_{\slice x}(E) \leq \rho_{\slice x}(E) \leq C$. 
If $\mathbb{1}_E$ is the characteristic function of Borel set $E\subset \R^2$, then
\begin{align*}
\int_{\R} \int_{\ell_x} \mathbb{1}_E\,\dd\rho_{\slice x}\,\dd x %&= \int_{\R}\, \rho_{\slice x}(E)\,\dd x\\
&= \int_{\R} \int_{\hat \L} \left(\sum_{s \in I \cap \ell_x} \mathbb{1}_{E}(s) (\vu \cdot e_1)(s)^{-1} + \tfrac{1}{2} \sum_{s \in \{p_I, q_I\} \cap E} \mathbb{1}_{E}(s) (\vu \cdot e_1)(s)^{-1}\right)\,\dd\hat\sigma(I)\,\dd x.
\end{align*}
The integrand is absolutely convergent, so we can apply Fubini's theorem to say
\begin{align*}
\int_{\R} \int_{\R^2} \mathbb{1}_E\,\dd\rho_{\slice x}\,\dd x %&= \int_{\R}\, \rho_{\slice x}(E)\,\dd x\\
&=  \int_{\hat \L}\int_{\R} \left(\sum_{s \in I \cap \ell_x} \mathbb{1}_{E}(s) (\vu \cdot e_1)(s)^{-1} + \tfrac{1}{2} \sum_{s \in \{p_I, q_I\} \cap E} \mathbb{1}_{E}(s) (\vu \cdot e_1)(s)^{-1}\right)\,\dd x\, \dd\hat\sigma(I)\\
&= \int_{\vec \L}\int_{\R} \sum_{s \in \vec I \cap \ell_x} \mathbb{1}_{E}(s) (\vu \cdot e_1)(s)^{-1} \dd x\, \dd\vec\sigma(I)\\
&= \int_{\vec \L} \int_0^1 \mathbb{1}_E(\pi(\vec I,t)) |\vec I| \,\dd\vec{\sigma}(\vec I)\\
&= \int_{\Lambda} \mathbb{1}_E\,\dd\rho,
\end{align*}
as required.\\

It only remains to show that $\rho_{\slice x}$ is uniquely defined by \eqref{eq:SliceMeasureProperty}. The right-hand side of \eqref{eq:SliceMeasureProperty} is independent of the choice of $\rho_{\slice x}$, so the integral of $\rho_{\slice x}$ with respect to any continuous function is prescribed. The support of $\rho_{\slice x}$ must be contained in $\bar{\M} \cap \ell_x$, which is a one-dimensional interval. The monotone convergence theorem then implies that the measure of $\rho_{\slice x}$ is prescribed on all open sets (and in particular must be finite). This then implies equivalence for all Borel sets \cite[Lemma~7.1.2]{Bogachev07}.
%
%We now turn to part \ref{res:disintegration}. For any $\delta>0$ it is easy to see that 
%\[ \int_{\R^2} A \,\dd\rho = \int_{\R} \frac{1}{2\delta} \int_{L_{x,\delta}} A\,\dd\rho\,\dd x. \]
%By \eqref{eq:FiniteSliceBounded} and the fact that $\rho$ is compactly supported, we can bound uniformly over small delta that
%\[ \frac{1}{2\delta} \int_{L_{x,\delta}}A \,\dd\rho \leq C \|A\|_{L^\infty(\M)} \mathbb{1}(|x| < x^*) \]
%for some $x^*$ big enough. Therefore by part \ref{res:slice} and the dominated convergence theorem
%\[ \int_{\M} A \,\dd\rho = \int_{\R} \lim_{\delta\to0} \frac{1}{2\delta} \int_{L_{x,\delta}\cap \M} A\,\dd\rho\,\dd x = \int_{\R} \int_\M A\,\dd\rho_{\slice x}\,\dd x. \]
%
%To prove part \ref{res:disintegration}, we have that for any Borel set $E \subset \R^2$,
%\[ \int_{E} \,\dd\rho = \int_{\R} \frac{1}{2\delta} \int_{L_{x,\delta}\cap E} \,\dd\rho_{\slice x}\,\dd x. \]
%By \eqref{eq:FiniteSliceBounded} and the fact that $\rho$ is compactly supported, we can bound uniformly over small delta that
%\[ \frac{1}{2\delta} \int_{L_{x,\delta}\cap E} \,\dd\rho_{\slice x} \leq C \mathbb{1}(|x| < x^*)\]
%for some $x_*$ big enough. Therefore by the dominated convergence theorem,
%\[ \int_{E} \,\dd\rho = \int_{\R} \lim_{\delta \to 0} \frac{1}{2\delta}  
\end{proof}

\section{Decomposition of response}\label{s:Decomposition}

With these results in hand, we can begin to break apart the susceptibility \eqref{eq:Susceptibility}.

Recalling the definitions of the stable and unstable bundles in Section~\ref{s:Setup}, let the unstable projection operator be 
\[ \P^u(x) = \vu(x) l^u(x) \]
and the stable projection operator be 
\[ \P^s(x) = \vs(x) l^s(x) = \id - \P^u(x). \]

\begin{proposition}\label{p:Decomposition}
	The susceptibility coefficients \eqref{eq:Susceptibility} can be decomposed as 
	\begin{equation}
	\kappa_n = \kappa^s_n + \kappa^X_n + \kappa^\rho_n + \kappa^\P_n
	\end{equation}
	where
	\begin{align}
	\kappa^s_n &= \int_\M \nabla (A\circ f^n) \cdot \P^s X \,\dd \rho \label{eq:KappaS}\\
	\kappa^X_n &= -\int_\M (A \circ f^n)\, l^u (DX)(\vu) \,\dd \rho \label{eq:KappaX}\\
	\kappa^\rho_n &= \int_{\vec \L} \left((A \circ f^n)(q_{\vec I}) (l^u X)(q_{\vec I}) - (A \circ f^n)(p_{\vec I}) (l^u X)(p_{\vec I}) \right)\,\dd\vec\sigma(\vec I) \label{eq:KappaRho}\\
	\kappa^l_n &= - \int_{\vec \L} \left(\int_0^1 (A \circ f^n)(\pi(\vec I,t))\, \dd l^u(t) X(\pi(\vec I,t))\right)\, \dd\vec\sigma(\vec I) \label{eq:KappaL}.
	\end{align}
\end{proposition}
\begin{proof}
Using $\id = \P^s + \P^u$ we have that 
\begin{equation} \int_\M \nabla (A\circ f^n) \cdot X\, \dd\rho = \kappa^s_n + \int_\M \nabla (A\circ f^n) \cdot \P^u X\,\dd\rho. \label{eq:ResponseSUDecomposition}\end{equation}
Now, using $\P^u(x) = \vu(x) l^u(x)$ and that $\|\vu(x)\| \equiv 1$, we have
\[ \nabla (A\circ f^n) \cdot \P^u X = \vu \cdot \nabla (A\circ f^n) l^u X = \frac{\dd}{\dd \vu}(A \circ f^n) l^u X. \]
We can write the second part of (\ref{eq:ResponseSUDecomposition}) as
\[ \int_\M \nabla (A\circ f^n) \cdot \P^u X\,\dd\rho = \int_{\hat {\L}} \int_0^1 \left(\frac{\dd}{\dd \vu}(A \circ f^n) l^u X\right)(\pi(I,t))\, \dd t\, \dd\hat{\rho}(I), \]
using \eqref{eq:Disintegration} and that the conditional density along unstable manifolds is constant (the situation would not be seriously different if it was non-constant).
Then, since $\vu$ is a unit vector along $\vec I$, $\frac{\dd\phi}{\dd \vu}(\pi(\vec I,t)) = (\vu \cdot (q_{\vec I} - p_{\vec I}))^{-1} \frac{\dd}{\dd t} \phi(\vec \pi(\vec I,t))$ and so, using the definition of $\vec \sigma$ \eqref{eq:VecSigmaDefinition},
\[ \int_\M \nabla (A\circ f^n) \cdot \P^u X\,\dd\rho = \int_{\vec{\L}} \int_0^1 \frac{\dd}{\dd t}(A \circ f^n)(\vec \pi(\vec I,t))\, \left(l^u X\right)(\vec \pi(\vec I,t))\, \dd t\, \dd \vec{\sigma}(\vec I). \]

Doing integration by parts on the inner integral, the right-hand side becomes
\begin{align} \int_{\vec {\L}} \left( (A\circ f^n\ l^u X)(q_{\vec I}) - (A\circ f^n\ l^u X)(p_{\vec I}) - \int_0^1 (A \circ f^n)(\vec \pi(\vec I,t))\, \dd \left(l^u X\right)(\vec \pi(\vec I,t))\,\right)\dd \vec{\sigma}(\vec I), \label{eq:IntegrationByParts} \end{align}
where the inner integral is a Stieltjes integral. By Proposition~\ref{p:lLeafTotalVariation} below in Section~\ref{s:DecayLRho} and the fact that $X$ is $C^2$, we know the Stieltjes integral is well-defined and bounded. The first two terms in \ref{eq:IntegrationByParts} become $\kappa^\rho_n$.

We can expand the inner integrating term on the right-hand side of \eqref{eq:IntegrationByParts} via the product rule so that
\begin{equation} \dd (l^uX)(\vec \pi(\vec I,t)) = \dd l^u(\vec \pi(\vec I,t)) X(\vec \pi(\vec I,t)) + (l^u DX \vu)(\pi(\vec I,t))\, \vu \cdot (q_{\vec I}-p_{\vec I})\, \dd t. \label{eq:ProductRuleLX} \end{equation}
The first term in \eqref{eq:ProductRuleLX} gives $\kappa^l_n$ when separated into its own integral; by rewriting back in terms of the original SRB measure $\rho$, the second term in \eqref{eq:ProductRuleLX} gives $\kappa^X_n$, as required. 

Note that all the integrands in (\ref{eq:KappaX}--\ref{eq:KappaL}) are bounded, and by Proposition~\ref{p:SigmaFinite}, $\vec{\sigma}$ is a finite measure. As a result, the terms $\kappa^X_n, \kappa^\rho_n, \kappa^l_n$ are finite and our separation of them is valid.
% \[ (\dd l \P^s X + \dd l \P^u X)(\vec \pi(\vec I_{p,q},t)). \]
% The first term of this equation gives $\kappa^l_n$; the second term gives
% \[ - \int_{\hat{\Lambda}} \int_0^1 \left(A \circ f^n\, \dd l \P^u X\right)(\vec \pi(\vec I_{p,q},t)) \frac{\dd\hat{\rho}(I)}{r \cdot (q-p)}. \]
% Now, notice that $\P^u = r l$, and so 
% \begin{align*} \int_0^1 \left(A \circ f^n\, \dd l \P^u X\right)(\vec \pi(\vec I_{p,q},t)) &= \int_0^1 \left(A \circ f^n\, l X\, \dd l r\right)(\vec \pi(\vec I_{p,q},t)) \\
% 	&= - \int_0^1 \left(A \circ f^n\, l X\, l \dd r\right)(\vec \pi(\vec I_{p,q},t)),\\
% \end{align*}
%since $l r \equiv 1$ so $\dd (lr) = 0$. Then, because our maps are piecewise affine, $r$ is constant along unstable manifolds so this term is zero\footnote{If the map was non-linear this term would be non-zero, and there would be an additional term corresponding to variation in the density of $\rho$ along the unstable manifold. Similar to the unstable divergence of \cite{Ruelle97} these would combine to give an extra susceptibility function contribution whose integrand is smooth and to which Proposition~\ref{p:DecayOfCorrelations} could be applied.}.
\end{proof}

As usual in uniformly hyperbolic systems, the stable contribution $\kappa^s_n$ decays exponentially as $n \to \infty$:
\begin{proposition}\label{p:KappaS}
	There exists a constant $C$ such that for all $n \in \N$,
	\[ | \kappa^s_n | \leq C \| X\|_{L^\infty} \|A\|_{C^1} \mu^n. \]
\end{proposition}
\begin{proof}
	We have that $l^u E^s \equiv 0$, and that $E^s \subset C^s$ and $E^u \subset C^u$ are uniformly transverse. Since $\|\vu\|, l^u\vu \equiv 1$ this means that $\|l^u\|$ is bounded by some constant $C_l$. Hence, considered pointwise, $\|\P^s\| = \|\id - \vu l^u\| \leq 1 + C_l$. 
	Thus, using that $Df^n \P^s$ is in the stable cone for all $n$,
	\[ |\kappa^s_n| \leq \sup | \nabla(A \circ f^n) \P^s X | = \sup | (\nabla A) \circ f^n Df^n \P^s X | \leq \|A\|_{C^1} \mu^n (1 + C_l) \|X\|_{L^\infty}. \] 
	as required.
\end{proof}

\section{Decay of $\kappa^X_n$}\label{s:DecayX}

One of the contributions to the susceptibility functions, $\kappa^X_n$, decays exponentially in $n$ if the vector field $l^u DX \vu$ has decay of correlations, which obtains if it lies in the Banach space $\Ht$ from Proposition~\ref{p:DecayOfCorrelations}. In this section we will show this is the case:
 \begin{proposition}\label{p:KappaX}
% 	Let $\htop$ be the topological entropy of $f$, and suppose that $\htop \lambda^{-1} \mu < 1$. Then 
There exists a constant $C$ such that for all $n \in \N$, $X \in C^2$, $A \in C^1$,
 	\[ | \kappa^X_n - \kappa^X_\infty | \leq C \| X\|_{C^2} \|A\|_{C^1} \xi^n, \]
 	where 
 	\[ \kappa^X_\infty =  - \int_\M A\, \dd \rho \int_\M l^u (DX)(\vu) \,\dd \rho. \]
 \end{proposition}
Since $\rho$ lies in $\Ht$ (see the Appendix) the main step is to show that $l^u (DX)(\vu)$ is a bounded multiplier of functions in $\Ht$. Since $DX \in C^1$ this means we must show that $l^u$ and $\vu$ are bounded multipliers:
\begin{lemma}\label{l:UnstableDivXSobolevSpace}
	There exists $C$ such that for $W$ a $C^1$ tensor field,
	\[ \| l^u W \vu \phi \|_{\Ht} \leq C \| W \|_{C^1} \| \phi \|_{\Ht}. \]
\end{lemma}

We will need to characterise the regularity of the unstable vector and covector bundles, and in particular where and how big their jumps are. To do this, we will construct piecewise constant functions $\{\vv\bh\}_{\beta\in\N}$ and $\{\hat l\bg\}_{\alpha\in\N}$ which converge to covariant vector bundles as $\beta,\alpha \to \infty$. The following propositions achieve this. 

For concision let us define $\nu_\beta(x) := \mu_\beta(f^\beta(x)) \lambda_\beta(x)^{-1}$, which is bounded by $\mu^\beta \lambda^{-\beta}$.
\begin{lemma}\label{l:PushforwardDifferences}
	Suppose $\vv_{(0)}$ is a unit vector field on $\M$ such that $\vv_{(0)} \in \C^u$. Then, there exists a constant $C$ such that for all $\beta \in \N$,
	\begin{equation} \left| D_x f^\beta \vv_{(0)}(x) / \| D_x f^\beta \vv_{(0)}(x) \| - \vu(f^\beta(x)) \right| \leq C \nu_\beta(x). \label{eq:PushforwardDifferences} \end{equation}
\end{lemma}

\begin{proof}[Proof of Lemma~\ref{l:PushforwardDifferences}]
	Because $\vv_{(0)},\vu \in \C^u$, they are both transverse to stable directions: as a result, there exist constants $C, c$ independent of $a, b$ such that for some scalar field $\|k(x)\| \geq c$ we have $v(x) := a(x) - k(x) \vu(x)$ lies in the unstable cone and $\|v(x)\| \leq C$.
	Then, 
	\[ (D_xf^\beta \vv_{(0)})(x) = k(x) D_xf^\beta \vu(x) + D_xf^\beta v(x),\]
	and in particular,
	\begin{align*} \| D_xf^\beta \vv_{(0)}(x) -  k(x) Df^\beta \vu(x) \| &\leq \| D_xf^\beta v(x) \|\\
	& \leq |\mu_\beta(f^\beta(x))| \| v(x) \| \\
	&\leq C |\mu_\beta(f^\beta(x))| \leq C \mu^\beta.\end{align*}
	Furthermore
	\[ \| k(x) Df^\beta \vu(x)\| \geq c |\lambda_\beta(x)| \geq c \lambda^\beta. \]
	It is not hard to show that
	\[ \tfrac{\psi}{\| \psi \|} - \tfrac{\chi}{\|\chi \|} = \|\psi\|^{-1} \left(\psi-\chi + (\|\chi\| - \|\psi\|)\tfrac{\chi}{\|\chi\|}\right),  \]
	so using the reverse triangle inequality,
	\[ \left\|\tfrac{\psi}{\| \psi \|} - \tfrac{\chi}{\|\chi \|}\right\| \leq \|\psi\|^{-1} \left(\|\psi-\chi\| + \left|\|\chi\| - \|\psi\|\right|\right) \leq 2 \|\psi\|^{-1} \| \psi-\chi \|.  \]
	From this we have, recalling that $\vu(f^\beta(x)) = D_xf^\beta \vu(x) / \|D_xf^\beta \vu(x) \|$, that
	\begin{align*}
	 \left\| \frac{D_x f^\beta \vv_{(0)}(x)}{\| D_x f^\beta \vv_{(0)}(x) \|} - \vu(f^\beta(x)) \right\| &\leq 2 \| D_x f^\beta \vu(x) \| \| k(x)^{-1} D_xf^\beta v(x) \| \\
	 &\leq 2 c^{-1} C |\lambda^{-\beta}(x) \mu^\beta(f^\beta(x))|, \end{align*}
	as required.
\end{proof}

\begin{lemma}\label{l:RightIteratedDifferences}
	Suppose $\vv_{(0)}$ is a locally constant unit vector field on $\M \backslash \S$ such that $\vv_{(0)} \in \C^u$ and let $\vv\bh(f^\beta(x)) := D_x f^\beta \vv_{(0)}(x) / \| D_x f^\beta \vv_{(0)}(x) \|$ be unit vector fields on $\beta$ also. Then, $\vv\bh \to \vu$ uniformly, and there exists a constant $C$ such that
\begin{equation*} \left| \vv\bhh(x) - \vv\bh(x) \right| \leq 2C \nu_\beta(f^{-\beta}(x)) \end{equation*}
with $\vv\bh$ piecewise constant on all final cylinders of length $\beta$.
\end{lemma}
\begin{proof}
	From the previous lemma, 
	\[ \| \vv\bh(x) - \vu(x) \| \leq C \nu_\beta(f^{-\beta}(x)), \]
	and by applying the triangle inequality to $\vv\bh(x) - \vv\bhh(x)$ the equation holds. On a final cylinder of length $\beta$ the constant initial vector field $\vv_{(0)}$ is pushed forward by the Jacobian of $f^\beta$, which is constant on the cylinder, proving the last claim.
\end{proof}

Let us define a unit unstable covector bundle $\hat l^u \propto l^u$ such that $\| \hat l^u \| \equiv 1$ and $\hat l^u \cdot e_1 > 0$. A similar result holds for the left eigenfunctionals.
\begin{lemma}\label{l:LeftIteratedDifferences}
	Suppose $\hat l_{(0)}$ is a locally constant covector field on $\M \backslash f(\S)$ such that $\hat l_{(0)} \in \C^s$ and let $\hat l\bg(f^{-\alpha}(x)) := (D_x f^{-\alpha})* \hat l_{(0)}(x) / \| (D_x f^{-\alpha})* \hat l_{(0)}(x) \|$ be unit vector fields on $\M$ also. Then, $\hat l\bg \to \hat l^u$ uniformly and there exists a constant $C$ such that
	\begin{equation*} \left| \hat l\bg(x) - \hat l\bgg(x) \right| \leq 2C \nu_\alpha(x). \end{equation*}
	with $\hat l\bg$ piecewise constant on all initial cylinders of length $\alpha+1$.
\end{lemma}

The cylinders must be of length $\alpha+1$ rather than $\alpha$ as might be expected simply as a result of the fact that the stable cone is discontinuous across the critical line (see its definition \ref{eq:StableCone} and Figure~\ref{fig:Cones}).

The final ingredient needed to prove Lemma~\ref{l:UnstableDivXSobolevSpace} is the following complexity bound:
\begin{lemma}\label{l:TopoPressure}
	There exists $C$ and $\zeta < 1$ such that for all $n \in \N$
	\begin{equation}
	\sum_{\cyi\in \Sigma^n: \O^b_\cyi \neq \emptyset} \sup_{x \in \O^b_\cyi} \nu_n(x) \leq C \zeta^n.
	\end{equation}
\end{lemma}
\begin{proof}
	Since $|\O^b_\cyi|\leq 2^\cyi$, this holds simply if $2 \lambda^{-1} \mu < 1$, which is the case on all parameters $(a,b)$ we consider.
\end{proof}
Let us remark that, as is standard in thermodynamical formalisms, we would generically expect the topological pressure for the SRB measure potential to be zero, which is to say that
\[ \lim_{n \to \infty} \frac{1}{n} \log \sum_{\cyi\in \Sigma^n: \O^b_\cyi \neq \emptyset} \sup_{x \in \O^b_\cyi} | \lambda_n(x)^{-1} | = 0. \]
This would imply Lemma~\ref{l:TopoPressure} for $\zeta \in (\mu,1)$.

\begin{proof}[Proof of Lemma~\ref{l:UnstableDivXSobolevSpace}]
From Lemmas~\ref{l:RightIteratedDifferences}--\ref{l:LeftIteratedDifferences} we have that
\begin{align} \vu &= \vv_{(0)} + \sum_{\beta=1}^\infty \vv\bhh - \vv\bh, \label{eq:RightVectorSum}\\
\hat l^u &= \hat l_{(0)} + \sum_{\beta=0}^\infty \hat l\bhh - \hat l\bh,
\end{align}
for appropriate vector fields $\vv_{(0)} \in C^1(\M), \hat l_{(0)} \in C^1(\M \backslash f(\S))$, which we can choose to be piecewise constant on the connected components of their domains. This convergence occurs in $L^\infty$. 

Now, we expect that (again in $L^\infty$, recalling that $\hat l \vu, \hat l\bh \vv\bh \geq c$ for some $c$)
\[ l^u W \vu = \frac{\hat l^u W \vu}{\hat l^u \vu} = \lim_{\alpha,\beta\to\infty} t\bgh, \]
where
\[ t\bgh := \frac{\hat l\bg W \vv\bh}{\hat l\bg \vv\bh}. \]

We will now show that as $\alpha,\beta \to \infty$, the $t\bgh$, considered as multiplication operators on $\Ht$, has a unique, bounded limit. The way we will do that is by writing
\begin{align*} t_{(\hat\alpha,\hat\beta)} &= t_{0,0} + \sum_{\alpha=0}^{\hat\alpha-1} (t_{(\alpha+1,0)} - t_{(\alpha,0)}) + \sum_{\beta=0}^{\hat\beta-1} (t_{(0,\beta+1)} - t_{(0,\beta)}) \\
	& \quad + \sum_{\alpha = 0}^{\hat\alpha-1} \sum_{\beta = 0}^{\hat\beta-1} (t_{(\alpha+1,\beta+1)} - t_{(\alpha+1,\beta)} - t_{(\alpha,\beta+1)} + t\bgh) \end{align*}
and showing these series are exponentially convergent as multiplication operators on $\Ht$.

From Lemma~\ref{l:RightIteratedDifferences} we have that any for any $\cyi \in \Sigma^{\beta}$, our $\vv_{(k)}, k \leq \beta$ is constant on the final cylinder $\O^e_{\cyi}$. 
Similarly from Lemma~\ref{l:LeftIteratedDifferences}, for any $\cyj \in \Sigma^{\beta+1}$ we have that any $\hat l_{(k)}, k \leq \beta$ is constant on the final cylinder $\O^b_{\cyj}$. The functions $t\bgh$, being functions of $\hat l\bg, \vv\bh$ and a $C^1$ vector field $W$, can therefore be written as follows:
\begin{equation} t\bgh = \sum_{\cyi \in \Sigma^\beta} \sum_{\cyj \in \Sigma^{g+1}} \mathbb{1}_{\O^e_{\cyi} \cap \O^b_{\cyj}} t\bh |_{\O^e_{\cyi}\cap\O^b_{\cyj}}, \label{eq:tnCylinderSum}\end{equation}
where the restriction of $t\bgh$ in the summands is $C^1$ with a $C^1$ extension $t\bgh^{\cyi,\cyj}$ to the whole of $\M$.

Now,
\begin{equation} t_{(\alpha,\beta+1)} - t\bgh = \frac{\hat l\bg \vv\bh\, \hat l\bg W (\vv\bhh - \vv\bh) - \hat l\bg W \bh\, \hat l\bg (\vv\bhh - \vv\bh)}{\hat l\bg \vv\bhh\, \hat l\bg \vv\bh}. \label{eq:DeltaQt}\end{equation}
Since $\hat l\bg \vv\bh \geq c$ and $\|\hat l\bg\|, \|\vv\bh\| = 1$ for all $\alpha, \beta \in \N$, we have that, for a constant $C'$ independent of $\alpha, \beta, \cyi, \cyj$, the extension of $t_{(0,\beta+1)} - t_{(0,\beta)}$ has
\begin{align} \| t_{(0,\beta+1)}^{\cyi,\cyj} - t_{(0,\beta)}^{\cyi,\cyj} \|_{C^1} &\leq \frac{2\|W\|_{C^1} \sup_{x \in \O^b_\cyi} \| \vv\bhh - \vv\bh \|}{c^2}\notag\\
	& \leq C' \ \| W\|_{C^1} \sup_{x \in \O^b_\cyi} \nu_\beta(x). \label{eq:DeltatnNorm}\end{align}

Now, from \eqref{eq:tnCylinderSum}, we find
\[ \| (t_{(0,\beta+1)} -  t_{(0,\beta)}) \phi \|_\Ht \leq \sum_{\cyi \in \Sigma^{\beta+1}} \sum_{\cyj \in \Sigma} \left\| (t_{(0,\beta+1)} -  t_{(0,\beta)}) |_{N_{\cyi,\cyj}} \mathbb{1}_{\O^e_{\cyi} \cap \O^b_{\cyj}} \phi \right\|_{\Ht}.\]
From \eqref{eq:DeltatnNorm}, Lemma~\ref{l:FunctionMultipliers} and Lemma~\ref{l:CylinderMultipliers} we find 
\begin{align*} \|(t_{(0,\beta+1)} -  t_{(0,\beta)}) \phi \|_\Ht &\leq \sum_{\cyi \in \Sigma^{\beta+1}, \O^e_\cyi \neq \emptyset} \sum_{\cyj \in \Sigma, \O^b_\cyj \neq \emptyset} \sup_{x \in \O^b_\cyi} \nu_\beta(x) \| W\|_{C^1} C_\# \| \phi \|_\Ht,\end{align*}
from which Lemma~\ref{l:TopoPressure} implies that for some $\zeta < 1$,
\[ \| (t_{(0,\beta+1)} -  t_{(0,\beta)}) \phi \|_\Ht \leq C \zeta^\beta \| W \|_{C^1} \| \phi \|_\Ht. \]
Similarly, we find 
\[\|  (t_{(\alpha+1,\beta)} -  t_{(\alpha,0)}) \phi \|_\Ht \leq C \zeta^\alpha \| W \|_{C^1} \| \phi \|_\Ht. \]
Finally, we must bound the norm of the double differences $t_{(\alpha+1,\beta+1)} - t_{(\alpha+1,\beta)} - t_{(\alpha,\beta+1)} + t\bgh$. Using \eqref{eq:DeltaQt} we find that 
\begin{align*} & t_{(\alpha+1,\beta+1)} - t_{(\alpha+1,\beta)} - t_{(\alpha,\beta+1)} + t\bgh \\
	& \qquad = \frac{(\hat l\bgg \vv\bh\, \hat l\bgg - \hat l\bg \vv\bh\, \hat l\bg) W (\vv\bhh - \vv\bh) - (\hat l\bgg W \vv\bh\, \hat l\bgg - \hat l\bg W \vv\bh\, \hat l\bg) (\vv\bhh - \vv\bh)}{\hat l\bgg \vv\bhh\, \hat l\bgg \vv\bh\, \hat l\bg \vv\bhh\, \hat l\bg \vv\bh} \end{align*}
which implies that 
\begin{align} &\| t_{(\alpha+1,\beta+1)}^{\cyi,\cyj} - t_{(\alpha+1,\beta)}^{\cyi,\cyj} - t_{(\alpha,\beta+1)}^{\cyi,\cyj} + t_{(\alpha,\beta)}^{\cyi,\cyj} \|_{C^1(N_{\cyi,\cyj})} \notag \\
	&\qquad\leq \frac{8\|W\|_{C^1} \sup_{x \in \O^b_\cyi} \| \vv\bhh - \vv\bh \| \sup_{x \in \O^e_\cyj} \| \hat l\bgg - \hat l\bg \|}{c^4}\notag\\
	&\qquad \leq C' \ \| W\|_{C^1} \sup_{x \in \O^b_\cyi} \nu_\beta(f^\beta(x)) \sup_{x \in \O^b_\cyj} \nu_\alpha(x). \label{eq:DeltaDeltatnNorm}
	\end{align}
Following through the same argument we find that 
\[\|  (t_{(\alpha+1,\beta+1)} - t_{(\alpha+1,\beta)} - t_{(\alpha,\beta+1)} + t\bgh) \phi \|_\Ht \leq C \zeta^{\alpha+\beta} \| W \|_{C^1} \| \phi \|_\Ht, \]
which is absolutely convergent as $\alpha,\beta \to \infty$. Finally, it is by this point clear that $t_{(0,0)}$ is appropriately bounded and piecewise $C^1$, so that also
\[ \| t_{0,0} \phi \|_\Ht \leq C \| W \|_{C^1} \| \phi \|_\Ht. \]

Hence, we have that $t\bgh$ has a limit considered as a multiplication operator on $\Ht$, which by the usual argument must be $l^u W \vu$, and so there exists a constant $C$ such that 
\[ \| l^u W \vu \phi \|_\Ht \leq C \| \phi \|_\Ht,\]
as required.
\end{proof}

\begin{proof}[Proof of Proposition~\ref{p:KappaX}.]
	%	It is required to show that $l DX(r)$ lies in $W^{1-\epsilon,1}$ or similar. We will do this by showing that it lies in $BV$, which has an inclusion into this space (I assume).
	
	$DX$ is a $C^1$ tensor field if $X$ is $C^2$, and so from Lemma~\ref{l:UnstableDivXSobolevSpace},
	\[ \| l^u (DX) \vu \rho \|_\Ht \leq C \| DX \|_{C^1} \| \rho \|_\Ht \leq C' \| X \|_{C^2}. \]
	Furthermore, $l^u (DX) \vu$ is bounded and $\rho$ is a finite measure, so $l^u (DX) \vu \rho$ is a (signed) Borel measure.
	The lemma then follows from an application of Proposition~\ref{p:DecayOfCorrelations}.
\end{proof}

\section{Decay of remaining terms}\label{s:DecayLRho}

Of the remaining terms, $\kappa^l_n$ and $\kappa^\rho_n$ contain contributions from the jumps in $Df$ along the singularity sets $\S$. Our claim is that these contributions decay exponentially as we send $n \to \infty$. This relies on Conjecture~\ref{c:Full} about the behaviour of the slice measure $\condmeas$ on the singular set $\S$ as we push it forward under $f$. 

We will henceforth assume $C$ to be a constant dependent only on $f$.

\subsection{Decay of $\kappa^l_n$}

%This proposition is the result of jumps in $l$ emanating from the jump in $D f$ on the singularity set, and pushing the terms in (\ref{eq:KappaL}) forward to the critical line.
In this section we will prove the following proposition:
\begin{proposition}\label{p:KappaLSum}
	The term $\kappa^l_n$ can be written as
	\[ \kappa^l_n = - \sum_{m=0}^\infty \int_{\S} (A\circ f^{n-m}) \, (l^s X)\circ f^{-m} \,l^u(x_-) \vs(x_+) (e_1 \cdot \vu)\, \mu_m(y_+)\, \dd \condmeas. \]
\end{proposition}

Proposition~\ref{p:CriticalLineSelfIntersectionMeasure} shows that $l^u$ is uniquely defined at almost all points along $\ell_{\S}$. Our results would still carry through if it were not the case, but it makes notation easier, and furnishes an interesting contrast with the one-dimensional case, where for a certain (dense) set of parameters the critical point is pre-periodic. 

We also would like the following result:
\begin{proposition}\label{p:lLeafTotalVariation}
	There exists a constant $C$ such that for any line segment $I \subset \M$,
	\[ \int_{I} \| \dd l^u \| \leq C. \]
\end{proposition}
\begin{proof}[Proof of Proposition~\ref{p:lLeafTotalVariation}]
	First we will prove that this integral makes sense. Recall that $l^u = \hat l^u/(\hat l^u \vu)$. We start by proving this lemma for $\hat l^u$.
	
	Recall the definition of the $\hat l\bg$ in Lemma~\ref{l:LeftIteratedDifferences}. We will show that these form a uniformly bounded Cauchy sequence in the space of functions of bounded total variation on $I$, of which $\hat l^u$ is the natural limit.
	
	From Lemma~\ref{l:LeftIteratedDifferences}, $\hat l\bg$ is piecewise constant on cylinders of length $\alpha+1$. Consequently, the total variation of $\hat l\bg$ on $I$ is zero when $\alpha+1 \leq m$, and the difference $\hat l\bgg - \hat l\bg$ is bounded by
	\begin{align*} \int_I \left\|\dd (\hat l\bgg - \hat l\bg)\right\| &= \sum_{\pi(I,t) \in f^{-\beta}(S), \beta \leq \alpha+2} \left| (\hat l\bgg - \hat l\bg)(\pi(I,t_+)) - (\hat l\bgg - \hat l\bg)(\pi(I,t_-)) \right|\\
	& \leq \sum_{\pi(I,t) \in f^{-\beta}(S), \beta \leq \alpha+2} C \nu_\alpha(\pi(I,t_-)) + C \nu_\alpha(\pi(I,t_+)). \end{align*}
	Since $I$ is a line, the intersection between any cylinder and $I$ has one connected component from Proposition~\ref{p:CylinderIntersections}. Thus, each cylinder of length $\alpha+2$ contains no more than two points in this sum, and each cylinder of length $\alpha$ therefore no more than eight. Hence, we can bound this by
	\[ \int_I \|\dd (\hat l\bgg - \hat l\bg)\| \leq 8C \sum_{\cyj\in \Sigma^{\alpha}: \O^b_\cyj \cap I \neq \emptyset} \sup_{x \in \O^b_\cyj \cap I} | \nu_\alpha(x)| \leq C \zeta^p, \]
	using Lemma~\ref{l:TopoPressure} in the last inequality. Thus, $\hat l\bg$ is a Cauchy sequence converging to $\hat l^u$ in the space of total variation on $I$, and 
	\[ \int_I \| \dd \hat l^u \| \leq C \sup_{x \in I} \nu_m(y) \]
	for some $C$ independent of $I$.
	
	The same argument using Lemma~\ref{l:RightIteratedDifferences} holds for $\vu$, and since $(\hat l^u, \vu) \mapsto \hat l^u/(\hat l^u \vu)$ is uniformly Lipschitz, we get the same result for $l^u$.
\end{proof}

This can be combined with the following lemma, which bounds Stieltjes integrals with respect to $l^u$ of vector fields $B$, which may be relatively large in unstable directions:
\begin{lemma}\label{l:StieltjesdlStable}
	For any local unstable manifold $\vec I \in \vec{\L}$ and any continuous vector field $B$,
	\[ \left| \int_{\vec I} \dd l^u\, B \right| \leq \int_{\vec I} \| \dd l^u \|\, \sup_{\vec I} \| \P^s B \|. \]
\end{lemma}
We will use the following result to prove this:
\begin{lemma}\label{l:SneakingInPs}
	For any $\vec I \in \vec{\L}$, vector $v$ and $x,y,z \in \vec I$,
	\[ (l^u(x) - l^u(y)) v = (l^u(x) - l^u(y)) \P^s(z) v, \]
\end{lemma}
\begin{proof}
	Using the decomposition $\id = \P^u + \P^s$ we have for any vector $v$ and $x,y,z \in \vec I$ that
\[ (l^u(x) - l^u(y)) v = (l^u(x) - l^u(y))(\P^u(z) v + \P^s(z) v). \]
Now,
\[ (l^u(x) - l^u(y))\P^u(z) v = (l^u(x) \vu(z) - l^u(y) \vu(z)) l^u(z) v = (l^u(x) \vu(x) - l^u(y) \vu(y)) l^u(z) v = 0 \]
since $\vu$ is constant along local unstable manifolds, and $l^u\vu \equiv 1$.
\end{proof}
\begin{proof}[Proof of Lemma~\ref{l:StieltjesdlStable}]
	 We have that the Stieltjes integral can be bounded in terms of partitions $P$ of $\vec I$ as:
	\begin{align*} \left| \int_{\vec I} \dd l^u\, B \right| &\leq \sup_{P} \sum_{ [x^i,x^{i+1}] \in P} \sup_{z^i \in [x^i,x^{i+1}]} | (l^u(x^i) - l^u(x^{i+1})) B(z^i) |. \end{align*}
	By Lemma~\ref{l:SneakingInPs} we can say 
	\begin{align*}| (l^u(x^i) - l^u(x^{i+1})) B(z^i) | &= | (l^u(x^i) - l^u(x^{i+1})) \P^s(z^i) B(z^i) |\\
	& \leq \| l^u(x^i) - l^u(x^{i+1}) \| \| (\P^s B)(z^i) \|,\end{align*}
	so
	\begin{align*}
	\left| \int_{\vec I} \dd l^u\, B \right| &\leq \sup_{P} \sum_{[x^ i,x^{i+1}] \in P} \| l^u(x^i) - l^u(x^{i+1}) \| \sup_{\vec I} \| \P^s B \| \\
	&= \int_{\vec I} \| \dd l^u \| \sup_{\vec I} \| \P^s B \| \end{align*}
	as required.
\end{proof}

This allows us to prove Proposition~\ref{p:KappaLSum}:
\begin{proof}[Proof of Proposition~\ref{p:KappaLSum}]
	If we define for any segment $\vec J \subset \vec I \in \hat \L$ that
	\[ \psi_{m,n}(\vec J) := \int_0^1 \dd l^u(\pi(\vec J,t)) (Df^m X\, A \circ f^n)(f^{-m}(\pi(\vec J,t))), \]
	then we have that
	\begin{equation} \kappa^l_n = - \int \psi_{0,n}(\vec I) \frac{\dd \hat{\rho}(I_{p,q})}{\vu \cdot (q-p)}. \label{eq:KappaLAsPsi}\end{equation}
	
	Now, since $l^u(x) D_{f^{-1}(x)} f = (\lambda_1 l^u)(f^{-1}(x))$ and $Df$ and $\lambda_1$ are constant on $f^{-1}(\vec I)$,
	\begin{align*} \psi_{m+1,n}(\vec I) &= \int_0^1 (\dd l^u \lambda_1)(f^{-1}(\pi(\vec I,t))) (Df^m X\, A \circ f^n)(f^{-m-1}(\pi(\vec I,t)))\\
		& = \lambda_1(u_{\vec I})) \psi_{m,n}(f^{-1}(\vec I)) \end{align*}
		for an arbitrary choice of point $u_{\vec I} \in f^{-1}(\vec I)$.
		
	Because $\psi_{m,n}$ are Stieltjes integrals with respect to $l^u$, we can decompose over subsegments of $\vec I$. If $\vec I$ decomposes as $\vec J \cup \{s\} \cup \vec J'$ with $s := \pi(\vec I, t^*)$ the point of intersection between segments $\vec J, \vec J'$,
	\[ \psi_{m,n}(\vec I) = \psi_{m,n}(\vec J) + \Delta_{\vec I}[l^u](s) (Df^m X\, A \circ f^n)(f^{-m}(s)) + \psi_{m,n}(\vec J'), \]
	where $\Delta_{\vec I}[l^u](s) := l^u(\pi(\vec I,t^*_+)) - l^u(\pi(\vec I,t^*_-))$ is the jump in $l^u$ along (and in the direction of) $\vec I$ at $s$.
	
	From Lemma~\ref{l:SigmaForwardMap} we then have
	\begin{align} \int_{\vec \L} \psi_{m,n}(\vec I) \,\dd\vec\sigma(\vec I) &= \int_{\vec \L} \sum_{\vec J = \M_\pm \cap \vec I \neq \emptyset} \lambda_1(u_{\vec J}) \psi_{m,n}(f(\vec J)) \,\dd\vec\sigma(\vec I) \notag \\
		&= \int_{\vec \L} \sum_{\vec J = \M_\pm \cap \vec I \neq \emptyset} \psi_{m+1,n}(\vec J)\,\dd\vec\sigma(\vec I) \notag \\
		&= \int_{\vec \L}\left(\psi_{m+1,n}(\vec I) - \sum_{s \in \vec I \cap \ell_{\S}} \Delta_{\vec I}[l^u](s) (Df^m X\, A \circ f^n)(f^{-m}(s))\right)\dd\vec\sigma(\vec I). \label{eq:PsimnInduction}
	\end{align}
	
	By Proposition~\ref{p:lLeafTotalVariation} and Lemma~\ref{l:StieltjesdlStable}, for $\vec\sigma$-almost any $\vec I \in \vec \L$
		\begin{align*}
		|\psi_{m,n}(\vec I)| &\leq \int_I \|\dd l^u\| \sup_I \| \P^s (Df^m X\, A \circ f^n)(f^{-m}(x)) \|\\
		& \leq C \| A\|_\infty \sup_I \| Df^m \P^s X \|\\
		& \leq C \|A \|_\infty \mu^m \|X\|_\infty,
		\end{align*}
		which converges to zero as $m \to \infty$. From this, 
		\eqref{eq:KappaLAsPsi} and \eqref{eq:PsimnInduction} we can conclude inductively that 
		\[ \kappa^l_n = \int_{\vec \L} \sum_{m=1}^\infty \sum_{s \in \vec I \cap \ell_{\S}} \Delta_{\vec I}[l^u](s) (Df^m X\, A \circ f^n)(f^{-m}(s))\dd\vec\sigma(\vec I). \]
		Note that the inner sum will contain at most one point.
		
		The argument in Lemma~\ref{l:StieltjesdlStable} tells us that, when $\vec I$ points from $\M_-$ into $\M_+$, 
		\[ \Delta_{\vec I}[l^u](s) (Df^m X\, A \circ f^n)(f^{-m}(s)) = (l^u(s_+) - l^u(s_-)) \P^s(s_+) (Df^m X\, A \circ f^n)(f^{-m}(s)).\]
		Since $\P^s(s_+) D_{f^{-m}(s)}f^m = \vs(s_+) \mu_m(s_+) l^s(s_+)$ and $l^u \vs \equiv 0$, and using Lemma~\ref{l:CondMeasSigma}, we obtain the requisite result.
	\end{proof}

\subsection{Decay of $\kappa^\rho_n$}

\begin{proposition}\label{p:KappaRhoSum}
	The term $\kappa^\rho_n$ can be written as
%		\[ \kappa^\rho_n = \sum_{m=0}^\infty \int_{\S} (A\circ f^{n+m})\, l (Df^{-m}\P^u X) \circ f^{m}\, \dd \rho_\S \]
	\[ \kappa^\rho_n = \sum_{m=1}^\infty \int_{\S} (A\circ f^{n+m})\, (l^u X) \circ f^{m}\, (\lambda_m^{-1}(x_+) - \lambda_m^{-1}(x_-)) \, \vu \cdot e_1 \,\dd \condmeas \]
\end{proposition}

\begin{proof}
	Let
	\[ \chi_{m,n}(\vec I_{p,q}) = (l^u X \, A \circ f^{n})(f^m(q))  \lambda_m^{-1}(q) - (l^u X \, A \circ f^{n})(f^m(p))  \lambda_m^{-1}(p). \]
	From the definition of $\kappa^\rho_n$ in Proposition~\ref{p:Decomposition}, it is clear that
	\begin{equation} \kappa^\rho_n = \int_{\vec \L} \chi_{0,n}\, \dd \vec{\sigma}. \label{eq:KappaRhoChi}\end{equation}
	Using Lemma~\ref{l:SigmaForwardMap}, we have that
	\begin{align*}
	\int_{\vec \L} \chi_{m,n}\, \dd \vec{\sigma} &=
%	 \int_{\vec \L} \sum_{\vec J \in \D(\vec I)} \lambda_1(u_{J})^{-1}\left((A \circ f^{n}\, l X)(f^{m+1}(q_{\vec J})) \lambda_m^{-1}(q_{\vec J}) -\right. \\
	 &\qquad\qquad \left.(A \circ f^{n}\, l^u X)(f^{m+1}(p_{\vec J})) \lambda_m^{-1}(p_{\vec J}) \right)\,\dd\vec\sigma(\vec I).
	\end{align*}
	Setting $u_{\vec J} = f^{p}(p_{\vec J})$ we can combine $\lambda_1(u_{\vec J}) \lambda_m(p_{\vec J}) = \lambda_{m+1}(p_{\vec J})$, and similarly for $q_{\vec J}$, keeping note of the fact that the endpoints written $p_{\vec J}, q_{\vec J}$ are considered as limits of $\vec J$ going to the endpoints.
	
	The points $p_{\vec I}$ and $q_{\vec I}$ will be among the boundary points of the preimages of the descendants of a local unstable manifold $\vec I$. Another possible boundary point, counted once as a starting point and once as an endpoint, will be where $\vec I_{p,q}$ intersects the singular line $\ell_{\S}$. With the correct signs, this leaves us that 
	\begin{equation} \int_{\vec \L} \chi_{m,n}\, \dd \vec{\sigma} = \int_{\vec \L} \chi_{m+1,n} + \sum_{s \in \vec I \cap \ell_{\S}} (A \circ f^{n}\, l^u X)(f^{m+1}(s)) \Delta_{\vec I}[\lambda_{m+1}^{-1}](s)\, \dd \vec{\sigma}. \label{eq:ChiRecursion} \end{equation}
		Now $|\chi_{m,n}| \leq C \lambda^{-m} \| X\|_\infty \|A \|_\infty$ which converges uniformly to zero as $m \to \infty$, so we can combine \eqref{eq:KappaRhoChi} and \eqref{eq:ChiRecursion} to obtain
		\[ \kappa^\rho_n  = \sum_{m=1}^\infty \int_{\vec \L} \sum_{s \in \vec I \cap \ell_{\S}} (A \circ f^{n}\, l^u X)(f^{m+1}(s)) \Delta_{\vec I}[\lambda_{m+1}^{-1}](s)\, \dd \vec{\sigma} \]
	An application of Lemma~\ref{l:CondMeasSigma} then gives us the required result.
\end{proof}

\subsection{Putting terms together}

The following result is a simple application of Conjecture~\ref{c:Full}, summing up the terms in the previous two propositions, and rewriting $\kappa^\rho_\infty$ and $\kappa^l_\infty$ as above.
\begin{proposition}\label{p:KappaLRhoDecay}
	Under Conjecture~\ref{c:Full}, there exists $C > 0$, $c \in (0,1)$ such that 
	\begin{equation} | \kappa^\rho_n - \kappa^\rho_\infty | < C c^n \|A\|_{C^\Aord} \|X\|_{\mathcal{B}} \label{eq:KappaRhoDecay} \end{equation}
		and
	\begin{equation} | \kappa^l_n - \kappa^l_\infty | < C c^n\|A\|_{C^\Aord} \|X\|_{\mathcal{B}} \label{eq:KappaLDecay} \end{equation}
	where
	\[ \kappa^\rho_\infty = - \rho(A) \int_{\vec \L} \left((l^u X)(q_{\vec I}) - (l^u X)(p_{\vec I}) \right)\,\dd \vec \sigma(\vec I)\]
	and
	\[\kappa^l_\infty = - \rho(A) \int_{\vec \L} \int_{\vec I} \dd l^u\, X\, \dd\vec\sigma(\vec I)  \]
\end{proposition}
\begin{proof}[Proof of Proposition~\ref{p:KappaLRhoDecay}]
	Notice that by applying Proposition~\ref{p:KappaRhoSum} with $A \equiv 1$ we have that 
	\[ \kappa^\rho_\infty = \rho(A) \sum_{m=1}^\infty \int_{\ell_{\S}} (l^u X) \circ f^m\, (\lambda_m^{-1}(c_+) - \lambda_m^{-1}(c_-))\, \dd\rho_\S(c). \]
		We can then apply the main conjecture (\ref{eq:Conjecture1}) to the expression in Proposition~\ref{p:KappaRhoSum} with $B = l^u X$, $\Gamma = 1$, giving that
	\[ |\kappa^\rho_n - \kappa^\rho_\infty | \leq \sum_{m=1}^\infty C \|A\|_{C^\Aord} \|l^s X\|_{\mathcal{B}} c^n \theta^m \leq C\|A\|_{C^\Aord} \|X\|_{\mathcal{B}} (1-\theta)^{-1} c^n,\]
	giving (\ref{eq:KappaRhoDecay}).
	\\
	
	In a similar fashion, by applying Proposition~\ref{p:KappaLSum}, we have that
	\[ \kappa^l_\infty = - \rho(A) \sum_{m=0}^\infty \int_{\S} (l^s X)\circ f^{-m} \, l^u(x_-) \vs(x_+)\, \mu_m\, \dd \rho_\S. \]
	
	To prove (\ref{eq:KappaLDecay}) we can split the expression in Proposition~\ref{p:KappaLSum} into two parts: one with $m \geq n$, which we can bound trivially using the decay in $\mu_m$, and one with $m < n$, to which we apply (\ref{eq:Conjecture2}) with $B = l^s X$, $\Gamma(c) = - l^u(c_-) \vs(c_+)$ for $c \in \ell_{\S}$ (which can be extended away from the critical line e.g. by $\Gamma(x,y) = \Gamma(0,y)$). This gives 
		\begin{align*} |\kappa^l_n - \kappa^l_\infty | &\leq \sum_{m=1}^{n-1} C \|A\|_{C^\Aord} \|l^s X\|_{\mathcal{B}} \|  - l^u(x_-) \vs_+ \|_{\mathcal{B}} c^{n} \theta^m +\\
		&\qquad \sum_{m=n}^\infty 2 \|A\|_{\infty} \|l^s X\, l^u(x_-) \vs_+ \|_{\infty} \| \mu_m \|_\infty \rho_\S(1)\\
			& \leq C\|A\|_{C^\Aord} \| X \|_{\mathcal{B}} (1-\theta)^{-1} c^n + C\|A\|_{\infty} \| X \|_{\infty} \mu^{n},\end{align*}
		yielding the required exponential decay.
\end{proof}

Of course, it is necessary to show that the limit of $\kappa_n$ as $n \to \infty$ is in fact zero:

\begin{proposition}\label{p:KappaInfinity}
	%		\[ \kappa^\rho_n = \sum_{m=0}^\infty \int_{\S} (A\circ f^{n+m})\, l (Df^{-m}\P^u X) \circ f^{m}\, \dd \rho_\S \]
	\[ \kappa^X_\infty + \kappa^\rho_\infty + \kappa^l_\infty = 0. \]
\end{proposition}
\begin{proof}
	Apply Proposition~\ref{p:Decomposition} in reverse, setting $A \equiv 1$.
\end{proof}

At last, we can prove the main theorem.

\begin{proof}[Proof of Theorem~\ref{t:Formal}]
	From Propositions~\ref{p:KappaS},~\ref{p:KappaX} and~\ref{p:KappaLRhoDecay}, we have
	\[ | \kappa^s_n + (\kappa^X_n - \kappa^X_\infty) + (\kappa^l_n - \kappa^l_\infty) + (\kappa^\rho_n - \kappa^\rho_\infty) | \leq C c^n \| A\|_{C^\Aord} \| X \|_{C^2} \]
	for some $c<1$ and $C>0$. Then, applying Propositions \ref{p:Decomposition} and \ref{p:KappaInfinity}, we have what is required.
\end{proof}

\appendix
\section{Proofs of various propositions}
	\subsection{Proof of Proposition~\ref{p:DecayOfCorrelations}}
	\begin{proof}[Proof of Proposition~\ref{p:DecayOfCorrelations}]
		We claim the Lozi maps satisfy the conditions of \cite[Theorem~2.5]{BaladiGouezel10}. Most necessary properties already hold by assumption, but in particular, the growth of the number of singularity curves intersecting at a single point is polynomial, as each step creates a finite number of singularity curves, and $f$ is a homeomorphism onto its image. As a result, there exists a Banach space $\Ht$ such that the {\it transfer operator}
		\[ (\mathcal{L} \phi)(x) := \mathbb{1}_{x \in \im f} | \det Df^{-1}(x) | \phi(f^{-1}(x)) \]
		has spectral radius equal to $1$ and has only isolated spectrum in a ball of radius strictly less than one. The SRB measure $\rho$ is the unique $1$-eigenfunction of $\mathcal{L}$, with corresponding left eigenfunctional being total (Lebesgue) integral on $\M$. Because the map $f$ is mixing, this is the only spectrum on the unit circle, and the rest must therefore be contained in a ball of radius $\xi < 1$. 
		
		Hence, there exists a constant $C$ such that for all $n > 1$,
		\[ \left\| \mathcal{L}^n \phi - \rho \int \phi \right\|_\Ht \leq C \xi^n \| \phi \|_\Ht.  \]
		Supposing that ${\phi_m}_{m \in \N}$ are $C^\Aord$ test functions such that the $\phi_m\, \dd x$ converge in $\Ht$ to $\phi$, we have, integrating against a $C^1$ function $A$, that by a change of variable
		\[ \int A \mathcal{L}^n \phi_m\, \dd x = \int (A \circ f^n)\, \phi_m\, \dd x. \]
		It can be seen from the definition of $\Ht$ in \cite{BaladiGouezel10} that integrating against $A$ is a bounded functional in $\Ht$. We therefore recover in the limit as $m \to \infty$ that
		\[ \int A \mathcal{L}^n \phi = \int (A \circ f^n)\, \phi, \]
		and obtain the required bound.
	\end{proof}

%	\begin{proof}[Proof of Proposition~\ref{p:RhoSFinite}]
%		We have that on the singular line, $\hat \sigma \circ \Wul \propto (e_1 \cdot r) \condmeas$, where $\condmeas$ is the conditional measure of $\rho$ on the singular line $\ell_{\S}$. As $r$ is a unit vector field we have that $| e_1 \cdot r| \leq 1$ are uniformly bounded, so the problem reduces to the finiteness of $\condmeas$. Now, if $\varpi_1$ is the projection onto the first coordinate, then
%		\[ \int_\S \dd \condmeas = \lim_{\epsilon \to 0} (2\epsilon)^{-1} \int_{(-\epsilon,\epsilon) \times \R} \dd \rho = \frac{\dd \varpi^*_1 \rho}{\dd \mathrm{Leb}}, \]
%		which is to say the density with respect to Lebesgue of the measure projection of the SRB measure $\varpi_1^*\rho$. From \cite{Young85} we know that this is density is of bounded variation, and hence finite.
%	\end{proof}

\subsection{Proof of Proposition~\ref{p:CriticalLineSelfIntersectionMeasure}}
To prove Proposition~\ref{p:CriticalLineSelfIntersectionMeasure} we require the following result:

\begin{proposition}\label{p:CriticalLineSelfIntersections}
	The critical line $\ell_{\S}$ intersects all $f^n(\ell_{\S})$ transversely for $n \neq 0$.
\end{proposition}
\begin{proof}
	It is easy to check that the critical line $\ell_{\S}$ lies outside the stable cone, and $f^{n}\ell_{\S}$ lies properly inside the stable cone for all $n \leq -1$. For positive $n$ the result obtains by applying $f^n$ to $f^{-n}\ell_{\S}$ and $\ell_{\S}$.
\end{proof}
	
	\begin{proof}[Proof of Proposition~\ref{p:CriticalLineSelfIntersectionMeasure}]
		Suppose this were not true.
		
		The set of points whose forward orbit intersects $\ell_{\S}$ more than once is
		\begin{equation} \cup_{n = 1}^\infty f^n(\ell_{\S}) \cap \ell_{\S}. \label{eq:ForwardOrbitIntersectionSet} \end{equation}
		We know from Proposition~\ref{p:CriticalLineSelfIntersections} that $f^n(\ell_{\S})$ and $\ell_{\S}$ always intersect transversally. Because both these sets are piecewise curves of finite length, the number of intersection points is finite. Thus, the countable union of these intersections (\ref{eq:ForwardOrbitIntersectionSet}) must be countable. If the $\condmeas$-measure of (\ref{eq:ForwardOrbitIntersectionSet}) is positive, then $\condmeas$ must contain atoms. By \eqref{eq:RhoMidDef}, this means so too must $\hat{\sigma}$, and therefore so too must $\hat{\rho}$.
		
		Now, suppose $I \in \hat\L$ is an atom of $\hat{\rho}$, that is $\hat{\rho}(\{I\}) > 0$. By definition, this means that $\rho(I) > 0$ also. Let $s$ be any point in $I$, so $I = \Wul(s)$. Now, because $\rho$ is $f$-invariant and finite, the union
		\[ \cup_{n =0}^\infty f^{-n}(\Wul(s)) \]
		cannot be disjoint, and since by the properties of local unstable manifolds $f^{-n}(\Wul(s)) \subset \Wul(f^{-n}(s))$ , we must have that $f^{-n^*}(\Wul(s)) \subset \Wul(s)$ for some $n^* \in \N$. But then since $\rho$ conditioned on $\Wul(s)$ is proportional to the length measure, and $\rho(f^{-n^*}(\Wul(s))) = \rho(\Wul(s))$, we have that $f^{-n^*}(\Wul(s)) = \Wul(s)$ and so $\Wul(s) = f^{n^*}(\Wul(s)$. However, $f$ is a bijection that expands unstable manifolds, so this cannot happen. Thus we have a proof by contradiction.
	\end{proof}
%	Note that this implies that it holds for the backward orbit; one can prove the same thing for $\rho |_{\S}$-almost every $y \in \ell_{\S}$, where $\rho|_{\S}$ is the conditional measure of $\rho$ restricted to $\S$.

\subsection*{Acknowledgements}

This research has been supported by the European Research Council (ERC) under the European Union's Horizon 2020 research and innovation programme (grant agreement No 787304), as well as by the ETH Zurich Institute for Theoretical Studies.

The author would like to thank Viviane Baladi for suggesting the project and for her detailed comments on the manuscript at various stages.

\subsection*{Data availability statement}
The datasets generated and/or analysed during the current study are available from the author on reasonable request.

\subsection*{Conflict of interest statement}
The author has no conflicts of interest to declare.

\bibliographystyle{plain}
\bibliography{lozi}

\end{document}